\documentclass[twoside,centertags]{amsart}

\usepackage[cmtip,all]{xy}
\usepackage{amsmath}
\usepackage{amsfonts}
\usepackage{amssymb}
\usepackage{amsthm}
\usepackage{graphicx}
\usepackage{mathrsfs}
\usepackage{latexsym}
\usepackage[pdftex]{hyperref}
\usepackage{etoolbox}

\apptocmd{\sloppy}{\hbadness 10000\relax}{}{}

\DeclareMathAlphabet\mathbfcal{OMS}{cmsy}{b}{n}

\newtheorem{theorem}{Theorem}[section]
\newtheorem{corollary}[theorem]{Corollary}
\newtheorem{lemma}[theorem]{Lemma}
\newtheorem{proposition}[theorem]{Proposition}
\theoremstyle{definition}
\newtheorem{definition}[theorem]{Definition}
\theoremstyle{remark}
\newtheorem{remark}[theorem]{Remark}
\newtheorem{example}[theorem]{Example}

\newcommand{\sPre}{\mathrm{sPSh}}
\newcommand{\SPre}{\mathrm{sPSh}^{\Delta}}

\newcommand{\sSet}{\mathbf{sSet}}

\newcommand{\C}{\mathcal{C}}
\newcommand{\op}{\mathrm{op}}
\newcommand{\Sm}{\mathrm{Sm}}
\newcommand{\SSm}{\mathbf{Sm}}
\newcommand{\upperminus}{{\hspace{-1pt}\mbox{\fontsize{5}{5}\selectfont$($-$)$}}}

\newcommand{\Sing}{\mathrm{Sing}}
\newcommand{\Nis}{\mathrm{Nis}}
\newcommand{\mot}{\mathrm{mot}}
\newcommand{\aff}{\mathrm{aff}}

\newcommand{\M}{\mathbf{M}}
\newcommand{\N}{\mathbf{N}}

\newcommand{\A}{\mathbb{A}}
\newcommand{\ZZ}{\mathbb{Z}}
\newcommand{\GG}{\mathbb{G}}

\newcommand{\id}{\mathrm{id}}
\newcommand{\Hs}{\mathscr{H}}
\newcommand{\HH}{\mathcal{H}}
\newcommand{\Uc}{\mathcal{U}}

\newcommand{\ADelta}{\mbox{{$\Delta$ \hspace{-2.37ex}{\fontsize{7pt}{0pt}$\Delta$}}}}
\newcommand{\myrightleftarrows}[1]{\mathrel{\substack{\xrightarrow{#1} \\[-.5ex] \xleftarrow{#1}}}}

\DeclareMathOperator{\Spec}{Spec}
\DeclareMathOperator{\pr}{pr}

\DeclareMathOperator*{\hocolim}{hocolim}

\DeclareMathOperator{\hocofib}{hocofib}

\DeclareMathOperator{\ttau}{\bar\tau}

\makeatletter
\newcommand{\labitem}[2]{
\def\@itemlabel{(\textrm{#1})}
\item
\def\@currentlabel{{\rm #1}}\label{#2} }
\makeatother

\begin{document}

\title{Model topoi and motivic homotopy theory} 

\author{Georgios Raptis}

\address{\newline 
G. Raptis \newline 
Universit\"{a}t Regensburg, 
Fakult\"{a}t f\"{u}r Mathematik, 
93040 Regensburg, Germany}
\email{georgios.raptis@mathematik.uni-regensburg.de}

\author{Florian Strunk}

\address{\newline
F. Strunk \newline 
Universit\"{a}t Regensburg, 
Fakult\"{a}t f\"{u}r Mathematik, 
93040 Regensburg, Germany}
\email{florian.strunk@mathematik.uni-regensburg.de}

\thanks{The authors are supported by the SFB/CRC 1085 \emph{Higher Invariants} (Universit\"at Regensburg) funded by the DFG}

\begin{abstract}
Given a small simplicial category $\C$ whose underlying ordinary category is equipped with a Grothendieck topology $\tau$, we construct a model structure on the category of 
simplicially enriched presheaves on $\C$ where the weak equivalences are the local weak equivalences of the underlying (non-enriched) simplicial presheaves. We show that this 
model category is a $t$-complete model topos and describe the Grothendieck topology $[\tau]$ on the homotopy category of $\C$ that corresponds to this model topos. After we 
first review a proof showing that the motivic homotopy theory is not a model topos, we specialize this construction to the category of smooth schemes of finite type, which is 
simplicially enriched using the standard algebraic cosimplicial object, and compare the result with the motivic homotopy theory. We also collect some partial positive 
results on the exactness properties of the motivic localization functor.
\end{abstract}

\maketitle
\setcounter{tocdepth}{1}
\tableofcontents

\section{Introduction}

The motivic homotopy theory introduced by Morel and Voevodsky \cite{MV99} provides a convenient framework for a homotopy theory of schemes and has led to the introduction of methods from 
algebraic topology with many spectacular applications. The motivic homotopy theory is obtained from two localization processes on the category of simplicial (pre)sheaves on $\Sm_S$, the 
category of smooth schemes of finite type over a base scheme $S$. The \emph{Nisnevich localization} is concerned with imposing descent with respect to the Nisnevich covers and ties the category of simplicial presheaves with 
that of schemes, regarded as a Grothendieck site. The \emph{$\A^1$-localization} imposes $\A^1$-invariance on simplicial presheaves where $\A^1$ is henceforth the scheme that plays the role of 
an interval object. A (fibrant) motivic space is a simplicial presheaf which is $\A^1$-homotopy invariant and satisfies Nisnevich descent. One obtains a motivic space by iterating these 
two localization processes, infinitely often in general, as each one generally destroys the effect of the other. The intricate interaction between the two localization processes is one  
of the subtle points in the theory.

The first localization taken alone corresponds to a construction that is available and well known for general Grothendieck sites. Given a Grothendieck site $(\C,\tau)$, Jardine \cite{JardineSimplicialPresheaves, Ja} 
constructed a model structure on the category $\sPre(\C)$ of simplicial presheaves, called the \emph{local model structure}, whose weak equivalences are those morphisms which 
induce isomorphisms on the $\tau$-sheaves of homotopy groups. The notion of fibrant object in this local model category encodes the property of homotopical descent with respect to hypercovers \cite{DHI}. 
On the other hand, the second localization generalizes to categories where there is a notion of homotopy so that one can speak of homotopy invariant simplicial presheaves. Combining both 
types of structure has led to the notion of a \emph{site with an interval} as a foundational framework for motivic homotopy theory (see~\cite[2.3.1]{MV99} and \cite[2.2]{VoevodskyHomology}). 

In the case of schemes, the $\A^1$-localization can alternatively be encoded by considering the simplicial enrichment $\SSm_S$ of $\Sm_S$ from \cite{HS}. The homotopy theory of enriched 
simplicial presheaves $\SPre(\SSm_S)$ consists of $\A^1$-ho\-mo\-topy invariant objects and moreover, it is equivalent to the $\A^1$-localization of $\sPre(\Sm_S)$ (see Proposition~\ref{A1project}).
In other words, one of the localizations for the motivic homotopy theory can be skipped by encoding $\A^1$-invariance directly into the objects of the category $\SPre(\SSm_S)$.
Motivated by this example, we consider in this paper a mixed setup which combines \emph{descent} with respect to an ordinary Grothendieck topology with a \emph{simplicial enrichment}. More precisely,
the setup consists of a simplicial category $\C$ whose underlying ordinary category $\C_0$ is equipped with a Grothendieck topology $\tau$. We prove that the category of \emph{simplicially enriched} 
simplicial presheaves $\SPre(\C)$ admits a model structure where a 
morphism is a weak equivalence if it is a local weak equivalence when regarded as a morphism between (non-enriched) simplicial presheaves in $\sPre(\C_0)$ (see Theorem~\ref{local-simp}). We call the 
resulting model category, denoted $\SPre(\C)_{\Uc \tau}$, the \emph{$\Uc$-local model category} where $\Uc \colon \SPre(\C) \to \sPre(\C_0)$ is the forgetful functor. This type of homotopy theory is related to 
homotopy theories that arise from a site with an interval, but there are some interesting and important differences, too.
When applied to the simplicial category $\SSm_S$ with the Nisnevich topology $\Nis$, this construction gives a model category $\SPre(\SSm_S)_{\Uc \Nis}$ which is \emph{not} equivalent to the motivic homotopy theory
- the latter is obtained by a further (non-trivial) left Bousfield localization. 

One of the properties that the motivic homotopy theory fails to satisfy is that of being a model topos. The notion of a model topos was introduced and studied by Rezk~\cite{Re1} and 
To\"en--Vezzosi~\cite{TV} and forms the model categorical analogue of an ordinary Grothendieck topos. The definition of a model topos involves homotopical descent properties and the theory of model topoi is intimately connected with homotopical sheaf theory. An argument for the failure of the motivic homotopy theory to form a model topos was sketched in \cite{SO12}, but we will review it 
here too in some more detail (see Proposition~\ref{prop:motivicisnotatopos}). This fact can be considered as a residual effect of the complications that arise when the Nisnevich and 
$\A^1$-localization processes are combined. Each of the two localizations taken separately does indeed define a model topos. The failure of this property for the motivic homotopy theory
implies in particular that the motivic localization functor does not preserve homotopy pullbacks in general. Based on results of Asok--Hoyois--Wendt~\cite{AsokHoyoisWendt} and 
Rezk~\cite{Rezkpistar}, we prove a positive result which says that a homotopy pullback whose lower right corner is $\pi_0^{\aff}$-$\A^1$-local (see Definition~\ref{pi0a1local}) is also a motivic 
homotopy pullback (see Theorem~\ref{theoremonmotivichomotopypullbacks}). 

On the other hand, the $\Uc$-local model category $\SPre(\C)_{\Uc \tau}$ is a model topos (see Theorem \ref{local-topos}). In particular, $\SPre(\SSm_S)_{\Uc \Nis}$ is a model topos. As in 
classical topos theory, there is a close connection between model topoi, defined as homotopy left exact left Bousfield localizations of enriched simplicial presheaves, and Grothendieck 
topologies. This was explored and studied in detail by To\"en--Vezzosi \cite{TV} for simplicial categories and by Lurie \cite{htt} for $\infty$-categories. In these homotopical contexts, a Grothendieck topology on a simplicial category 
(or $\infty$-category) $\C$ corresponds to an ordinary Grothendieck topology $\ttau$ on the homotopy category of $\C$. We emphasize that this differs from our basic setup where the simplicial enrichment 
and the Grothendieck topology are independent of each other. To\"en--Vezzosi \cite{TV} proved the existence of local model structures associated with a simplicial category $\C$ equipped 
with a Grothendieck topology $\ttau$ in this homotopical sense. This \emph{local model category}  $\SPre(\C, \ttau)$ is a model structure on the category of enriched simplicial presheaves $\SPre(\C)$ where the weak equivalences are those morphisms 
which induce isomorphisms on the \mbox{$\ttau$-sheaves} of homotopy groups (see Theorem~\ref{toen-vezzosi}). 
Moreover, To\"en and Vezzosi proved that this construction recovers all (t-complete) model topoi (see Theorem \ref{toen-vezzosi2}). Thus, the (t-complete) model topos $\SPre(\C)_{\Uc \tau}$ also arises in this way from a Grothendieck topology $[\tau]$ on $\mathrm{Ho}(\C)$. We study this induced Grothendieck topology and compare it with $\tau$ (see Subsection \ref{comparinggrothendiecktopologies}). 
Then we specialize this comparison to the case of $\SSm_S$ equipped with the Nisnevich topology and give an interpretation as to what type of descent, necessarily weaker than Nisnevich descent, is encoded in the $\Uc$-local model topos $\SPre(\SSm_S)_{\Uc \Nis}$.
While this particular $\Uc$-local model topos and its connection with the motivic homotopy theory is our main motivation for considering $\Uc$-local model structures in this paper, the general construction 
may be useful for a comparative study also in other contexts where there are two localization processes in interaction, one for descent and one for homotopy invariance. For example, the study of two such localization processes is also
central in the context of differential cohomology (see \cite{BTV}).

\medskip 

The paper is organized as follows. In Section~\ref{section:modeltopoi}, we review the theory of model topoi and discuss some of their properties. 
In Section \ref{local-model-structures}, we prove the existence of the $\Uc$-local model structure on $\SPre(\C)$ and show that it is a model topos (Theorems \ref{local-simp} and \ref{local-topos}). 
In Subsection~\ref{comparinggrothendiecktopologies}, we identify the associated topology $[\tau]$ on the homotopy category of $\C$ that corresponds to this model topos, and discuss the comparison 
between the $\tau$- and $[\tau]$-sheaf conditions.

In Section \ref{motivic-spaces}, we recall from \cite{HS} the simplicial enrichment of the category $\Sm_S$  that is defined by the standard algebraic cosimplicial object.
We show that the $\A^1$-localization of the projective model category $\sPre(\Sm_S)$ is Quillen equivalent to the projective model category $\SPre(\SSm_S)$ on enriched simplicial 
presheaves (Proposition~\ref{A1project}). Thus, it defines a model topos - even though $\A^1$-localization is not homotopy left exact. Then we recall the definition of (several known 
models for) the motivic homotopy theory and prove that it is not a model topos (Subsection~\ref{motivicnotatopos}). In Subsection \ref{partialinteractionofsingandnis}, we collect some 
positive results on the exactness properties of motivic localization. 

The $\Uc$-local model structure $\SPre(\SSm_S)_{\Uc \Nis}$ and its relationship with the motivic homotopy theory are discussed in Section~\ref{sectionfive}.
We construct a useful Quillen equivalent model for this $\Uc$-local model category using non-enriched simplicial presheaves (Theorem~\ref{theorem:singlocalmodel}). 
We also discuss the Grothendieck topology $[\Nis]$ on $\mathrm{Ho}(\SSm_S)$ that is associated with $\SPre(\SSm_S)_{\Uc \Nis}$ and explain the difference 
between Nisnevich descent and $[\Nis]$-descent (Subsection~\ref{comparison-with}). Finally, Subsection \ref{summary} contains a diagram which summarizes the different 
model categories and Quillen adjunctions that arise in the case of $\SSm_S$ equipped with the Nisnevich topology.

\section{Model Topoi}\label{section:modeltopoi}

\subsection{Projective model structures} 

Let $\sSet$ denote the simplicial model category of simplicial sets.
Let $\C$ be a small simplicial (i.e., simplicially enriched) category and let $\SPre(\C)$ denote the functor category of simplicial functors $\C^{\mathrm{op}} \to \sSet$.
A morphism $\eta\colon F \to G$ in $\SPre(\C)$ is an \emph{objectwise weak equivalence} (respectively, \emph{objectwise fibration}) if for every $c \in \mathrm{Ob}(\C)$, the map $\eta_c\colon F(c) \to G(c)$ is a weak equivalence (respectively, fibration) of simplicial sets.
A morphism $\eta\colon F \to G$ is a \emph{projective cofibration} if it has the left lifting property with respect to all morphisms which are objectwise weak equivalences and fibrations. 
The category $\SPre(\C)$ is enriched, tensored and cotensored over $\sSet$ with the (co)tensor structure defined objectwise using the simplicial structure of the category of simplicial sets. 
The following theorem is well known.

\begin{theorem} \label{proj-simp}
The classes of projective cofibrations, objectwise weak equivalences and objectwise fibrations define a proper simplicial 
combinatorial model structure on the category $\SPre(\C)$.
\end{theorem}

This model category is called the \emph{projective model category}.
We recall a precise definition of sets of generating cofibrations and trivial cofibrations.
A set of generating cofibrations is defined by the morphisms 
\[
\mathrm{map}_{\C}(-, c) \times \partial \Delta^n \hookrightarrow \mathrm{map}_{\C}(-, c) \times \Delta^n,
\]
for every $c \in \mathrm{Ob}(\C)$ and $n \geq 0$, and a set of generating trivial cofibrations is defined by the morphisms 
\[
\mathrm{map}_{\C}(-, c) \times \Lambda^k_n \hookrightarrow \mathrm{map}_{\C}(-, c) \times \Delta^n,
\]
for every $c \in \mathrm{Ob}(\C)$, $n>0$, and $0 \leq k \leq n$. 
The model category is lifted from the product (cofibrantly generated) model category $\prod_{\mathrm{Ob}(\C)} \sSet$ along the simplicially enriched (Quillen) adjunction 
\[
i_{!}: \sSet^{\mathrm{Ob}(\C)} \rightleftarrows \sSet^{\C^{\mathrm{op}}}: i^*
\]
where $i^{*}$ is the restriction functor along the inclusion $i\colon \mathrm{Ob}(\C) \to \C^{\mathrm{op}}$. 

\medskip

By regarding a set as a constant simplicial set, a small ordinary category $\C$ can be considered as a (discrete) simplicially enriched category where the mapping spaces are
constant simplicial sets. In this case, the category $\SPre(\C)$ is just the category of ordinary simplicial presheaves, denoted $\sPre(\C)$, and the model structure in Theorem~\ref{proj-simp} 
is the standard projective model structure. On the other hand, any simplicial category $\C$ has an underlying ordinary category $\C_0$, obtained by forgetting the simplicial 
enrichment. We emphasize the simplicial enrichment of $\C$ in the notation $\SPre(\C)$ because we are interested in the comparison between the projective model categories 
$\SPre(\C)$ and $\sPre(\C_0)$ and their left Bousfield localizations.
There is a Quillen adjunction
\begin{equation}\label{eq:adjunctionHandU}
\Hs: \sPre(\C_0) \rightleftarrows \SPre(\C): \Uc 
\end{equation}
where $\Uc$ denotes the forgetful functor and $\Hs$ is the colimit-preserving (simplicially enriched) Kan extension of the functor
\[
\Hs_{|\C}: \C_0 \to \SPre(\C), \ c \mapsto \mathrm{map}_{\C}(-, c).
\]
We note that the right adjoint $\Uc$ preserves colimits.

\subsection{Small presentations}

We denote by $\M_S$ the left Bousfield localization of a left proper combinatorial model category $\M$ at a set of morphisms $S$.
We recall that this localized model category always exists in the context of combinatorial model categories (see~\cite[A.3.7]{htt}).
The model category $\M_S$ is again cofibrantly generated and left proper.
It is also simplicial if $\M$ is.
The weak equivalences (respectively, fibrations) in $\M_S$ are called $S$-local equivalences (respectively, $S$-local fibrations). 

\begin{definition}
A \emph{small presentation} $(\C, S)$ consists of a small simplicial category $\C$ and a set of morphisms $S$ in $\SPre(\C)$.
A \emph{small presentation of a model category $\M$} is a triple $(\C, S, F)$ where $(\C, S)$ is a small presentation and $F$ is the left adjoint of a Quillen equivalence 
\[
F: \SPre(\C)_S \rightleftarrows \M : G.
\]
A model category $\M$ is called \emph{presentable} if it has a small presentation. 
\end{definition}

Every presentable model category has a small homotopically dense subcategory of homotopically presentable objects.
Therefore, not every model category can be presentable.
For example, discrete model categories which do not have a small dense subcategory provide examples of non-presentable model categories.
The following theorem of Dugger \cite{Dugger} identifies a large class of presentable model categories (see also \cite{RR2}).

\begin{theorem}[Dugger \cite{Dugger}] \label{dugger}
Every combinatorial model category is presentable.
\end{theorem}

\begin{remark}
The definition of a small presentation in \cite{Dugger} requires that $\C$ is an ordinary category.
Our definition of a presentable model category is therefore seemingly more general than the definition in \cite{Dugger} - ours allows $\C$ to be a non-discrete simplicial category.
However, as the model category $\SPre(\C)$ is always combinatorial, Dugger's theorem shows that it admits a small presentation defined by an ordinary category.
Hence, the two definitions are equivalent. 
\end{remark}

\begin{remark} 
The property of being presentable is invariant under Quillen equivalences.
If $\M$ is presentable and $F\colon \mathbf{N} \to \M$ is a left Quillen equivalence, then $\mathbf{N}$ admits a small presentation as well (see \cite[Prop.~5.10, Cor.~6.5]{Dugger2}). 
\end{remark}

\subsection{Model topoi}
We review the basic theory of model topoi as introduced by Rezk \cite{Re1} and To\"en--Vezzosi \cite{TV}.
Using the correspondence between presentable model categories and presentable $\infty$-categories, this theory 
is the model categorical counterpart of $\infty$-topos theory as developed by Lurie \cite{htt}. 

A left Quillen functor $F\colon \M \to \N$ is called \emph{homotopy left exact} if it preserves finite homotopy limits.
The proof of the following proposition is straightforward.

\begin{proposition}
Let $\M$ be a left proper combinatorial model category, $T$ a set of morphisms in $\M$, and $S$ a set of $T$-local equivalences.
Consider the left Bousfield localizations
\[
\id_T\colon \M \xrightarrow{\id_S} \M_S \xrightarrow{\id_{T/S}} \M_T.
\]
\begin{itemize}
  \item[(a)] If $\id_S$ and $\id_{T/S}$ are homotopy left exact, then so is $\id_T$.
  \item[(b)] If $\id_T$ is homotopy left exact, then so is $\id_{T/S}$.
\end{itemize}
\end{proposition}

\begin{definition} \label{def-model topos}
A small presentation $(\C, S)$ is called a \emph{model site} if the left Quillen functor 
\[
\id_S : \SPre(\C) \to \SPre(\C)_S
\]
is homotopy left exact.
A model category $\bf{M}$ is called a \emph{model topos} if it is Quillen equivalent to $\SPre(\C)_S$ for some model site $(\C, S)$. 
\end{definition}

We have the following useful criterion for a small presentation $(\C, S)$ to define a model site.

\begin{proposition} \label{criterion}
The left Quillen functor $\id_S \colon\SPre(\C) \to \SPre(\C)_S$ is homotopy left exact if and only if the class of $S$-local equivalences is closed under homotopy pullbacks in $\SPre(\C)$.
\end{proposition}
\begin{proof}
See \cite[Prop.~5.6]{Re1}, \cite[Prop.~6.2.1.1]{htt}.
\end{proof}

We recall an intrinsic characterization of model topoi in terms of descent properties which is due to Rezk \cite{Re1}.

\begin{definition} \label{descent}
We say that a model category $\M$ satisfies \emph{homotopical descent} if given the following data:
\begin{enumerate}
\item[(a)] a small category $I$,
\item[(b)] $Y \colon I^{\triangleright} \to \bf{M}$ a homotopy colimit diagram, where $I^{\triangleright}$ denotes 
the category $I$ with an added terminal object $\infty \in I^{\triangleright}$, 
\item[(c)] $X\colon I^{\triangleright} \to \bf{M}$ a functor,
\item[(d)] $\phi\colon X \to Y$ a natural transformation such that for every $i \to j$ in $I$ the diagram 
\[
\xymatrix{
X(i) \ar[r] \ar[d] & X(j) \ar[d] \\
Y(i) \ar[r] & Y(j) 
}  
\]
is a homotopy pullback,
\end{enumerate}
then the following hold:
\begin{enumerate}
\labitem{HD1}{HD1} If for every $i \in \mathrm{Ob} I$ the diagram
\[
\xymatrix{
X(i) \ar[r] \ar[d] & X(\infty) \ar[d] \\
Y(i) \ar[r] & Y(\infty) 
}  
\]
is a homotopy pullback, then $X$ is a homotopy colimit diagram. 
\labitem{HD2}{HD2} If $X$ is a homotopy colimit diagram, then the diagram  
\[
\xymatrix{
X(i) \ar[r] \ar[d] & X(\infty) \ar[d] \\
Y(i) \ar[r] & Y(\infty) 
}  
\]
is a homotopy pullback for every $i \in \mathrm{Ob} I$.
\end{enumerate}
\end{definition}

\begin{example}[Mather's second cube theorem (see~{\cite[Thm.~25]{Mather}})]\label{mather}
Suppose that a model category $\M$ satisfies \eqref{HD1}.
Consider a cube in $\M$
\[
\xymatrix@=0.5cm{
& A\ar@{->}[rr]\ar@{-}[d]\ar@{->}[dl] & & B\ar@{->}[dd]\ar@{->}[dl]\\
C\ar@{->}[dd]\ar@{->}[rr]&\ar@{->}[d]&D\ar@{->}[dd]&\\
& A' \ar@{-}[r]\ar@{->}[dl]&\ar@{->}[r]& B'\ar@{->}[dl]\\
C'\ar@{->}[rr]&&D'&\\
}
\]
where the bottom face is a homotopy pushout and all the side faces are homotopy pullbacks.
Then the top face is a homotopy pushout. 
\end{example}

\begin{example}\label{puppe}
Let $\M$ be a model category which satisfies \eqref{HD1}. Let $X\in\M$ be a pointed object in $\M$ and let $F, E\colon I\to \M$ be two diagrams in $\M$ such that there are natural transformations $F \to E \to cX$ with the property that
\[
F(i)\to E(i)\to X 
\]
is a homotopy fiber sequence for all $i\in I$.
Then also
\[
\underset{i\in I}{\hocolim}~F(i)\to \underset{i\in I}{\hocolim}~E(i)\to X 
\]
is a homotopy fiber sequence.
To see this, let us suppose for simplicity that $\M$ is cofibrantly generated and $E$ is a cofibrant-fibrant diagram in the projective model category $\M^I$.
Then consider the solid diagram
\[
\xymatrix{
&F(i) \ar@{-->}[r]\ar@{-->}[ld] & \ldots \ar@{-->}[r]& F(j)\ar@{-->}[ld]\ar@{-->}[r]& A\ar@{->>}[ld]\ar@{->}[d]\\
E(i) \ar@{->}[r]\ar@{->}[drr]&\ldots \ar@{->}[r]& E(j)\ar@{->}[d]\ar@{->}[r]& hE\ar@{->}[d]& {\bar *}\ar@{->>}[dl]\\
	       &                 &       X    \ar@{=}[r]    &   X             &&
}
\]
where $hE$ denotes the (homotopy) colimit, $\bar *$ is obtained by a factorization $*\xrightarrow{\sim} \bar *\twoheadrightarrow X$ and $A$ is the pullback of $hE\to X$ along $\bar *\twoheadrightarrow X$.
We may assume that $X$ and $hE$ are fibrant, so this pullback is the homotopy fiber of $hE\to X$.
Let $F\colon I\to \M$ be the diagram defined by the pullbacks of $E(i)\to hE$ along $A\to hE$.
These pullbacks are also homotopy pullbacks, hence $F(i)$ is a model for the homotopy fiber of $E(i)\to X$.
Then the claim follows as an application of \eqref{HD1}. 
\end{example}

\begin{example}\label{fiberseq}
Let $\M$ be a model category which satisfies \eqref{HD1}.
Let $X$ and $Y$ be pointed objects of $\M$ and consider a homotopy fiber sequence $F\to E\to X$.
Then we have:
\begin{enumerate}
 \item\label{ff1} $\Sigma(X\times Y)\simeq (X*Y)\vee \Sigma X\vee \Sigma Y$.
 \item\label{ff2} There is a homotopy fiber sequence $\Sigma\Omega X\to X\vee X\to X$.
 \item\label{ff3} There is a homotopy fiber sequence $\Omega X*\Omega Y\to X\vee Y\to X\times Y$.
 \item\label{ff4} There is a homotopy fiber sequence $F*\Omega X\to \hocofib(F\to E)\to X$.
 \item\label{ff5} There is a homotopy fiber sequence $\Omega X*\Omega X\to \Sigma\Omega X\to X$.
\end{enumerate}
Here $\Sigma$ denotes $S^1\wedge -$ and all functors are assumed to be derived. These statements are consequences of Mather's second cube theorem (Example~\ref{mather}), as explained 
in \cite{Doeraene}, and Example~\ref{puppe}.
The authors in op.cit.\ consider the cube theorem as an axiom and study its consequences. More precisely, assertion \eqref{ff1} is \cite[Cor.~2.13]{Doeraene} and \eqref{ff2} 
follows directly from Example~\ref{puppe} applied to the diagram $X \leftarrow * \rightarrow X$ over $X$. 
The statement \eqref{ff3} is \cite[Prop.~4.6]{Doeraene} and \eqref{ff4} is \cite[Cor.~4.3]{Doeraene}.
Finally, \eqref{ff5} follows from \eqref{ff4} applied to the fiber sequence $\Omega X\to *\to X$. 
\end{example}

\begin{example} [groupoids are effective] \label{effective groupoids}
Let $\mathbf{Top}$ be the standard model category of topological spaces. It is classically known that $\mathbf{Top}$ satisfies homotopical descent.
As an instance of \eqref{HD2}, let $X_{\bullet}$ be a Reedy cofibrant simplicial space such that for each $u\colon [n] \to [m]$, the 
square
\[
 \xymatrix{
 X_{m+1} \ar[d]_{d_{m+1}} \ar[r]^{\bar{u}^*} & X_{n+1} \ar[d]^{d_{n+1}} \\
 X_m \ar[r]^{u^*} & X_n
 }
\]
is a homotopy pullback.
Here $\bar{u}(i) = u(i)$ for $i \leq n$ and $\bar{u}(n+1) = m+1$.
Then the square 
\[
 \xymatrix{
 X_1 \ar[r]^{d_0} \ar[d]_{d_1} & X_0 \ar[d] \\
 X_0 \ar[r] & |X_{\bullet}|
 }
\]
is also a homotopy pullback.
Similar assertions hold for more general model categories satisfying \eqref{HD2}. 
\end{example}

\begin{example}
Since $\sSet$ satisfies \eqref{HD1} and \eqref{HD2}, so do also the model categories $\SPre(\C)$ for any small 
simplicial category $\C$.
It is easy to see that these properties are invariant under homotopy left exact Bousfield localizations.
Therefore every model topos satisfies homotopical descent. 
\end{example}

\begin{theorem}[Rezk~\cite{Re1}]\label{Giraud}
A presentable model category is a model topos if and only if it satisfies homotopical descent.
\end{theorem}
\begin{proof}
See \cite[Thm. 6.9]{Re1}.
\end{proof}

\begin{remark}
There is an analogue of this characterization as well as a Giraud-type theorem for $\infty$-topoi in \cite[Thm.~6.1.0.6]{htt}.
In the setting of model categories, Giraud theorems are also obtained by Rezk \cite{Re1} and To\"en--Vezzosi \cite{TV}.
\end{remark}

\begin{example}[disjoint coproducts] \label{disjoint-coproducts}
Let $\M$ be a model topos, $0$ denote the initial object, and $Y, Z$ be cofibrant objects in $\M$.
Then the (homotopy) pushout square 
\[
\xymatrix{
0 \ar[r] \ar[d] & Y \ar[d] \\
Z \ar[r] & Y \bigsqcup Z
}
\]
is also a homotopy pullback.
The proof is analogous to \cite[Prop.~6.1.3.19(iii)]{htt} or can easily be derived directly from Definition \ref{def-model topos}. 
\end{example}

\subsection{Forcing model topoi via localization}

Let $\M$ be a model topos and $S$ a set of morphisms in $\M$. 
While the left Bousfield localization $\M_S$ is not a model topos in general, there is a \emph{closest} model topos associated with $(\M, S)$. 
This is simply given by localizing further at the smallest class generated by the $S$-local equivalences which is closed under homotopy pullbacks in $\M$. 
The set-theoretical problem of the existence of this Bousfield localization can be solved similarly as for the analogous statement about $\infty$-topoi 
\cite[Prop.~6.2.1.2]{htt}. 

\begin{theorem}\label{thmconstructingmodeltopoi}
Let $\M$ be a model topos and $S$ a set of morphisms in $\M$.
Suppose that $\M$ is a left proper combinatorial model category.
Then there is a set of morphisms $\tilde{S}$ in $\M$ such that:
\begin{enumerate}
\item\label{thmconstructingmodeltopoi1} The class of $\widetilde{S}$-local equivalences contains the $S$-local equivalences.
\item\label{thmconstructingmodeltopoi2} The left Quillen functors
\begin{align*}
 \mathrm{id}_{\widetilde{S}} \colon   &{\bf{M}} \to {\bf{M}}_{\widetilde{S}}\\
 \mathrm{id}_{\widetilde{S}/S} \colon & \M_S \to \M_{\widetilde{S}}
\end{align*}
are homotopy left exact.
As a consequence, $\M_{\widetilde{S}}$ is again a model topos.
\item For every other set of morphisms $T$ in $\M$ satisfying \eqref{thmconstructingmodeltopoi1}-\eqref{thmconstructingmodeltopoi2}, 
the functor $\mathrm{id}_{T/\widetilde{S}} \colon \M_{\widetilde{S}} \to \M_{T}$ is a homotopy left exact left Quillen functor. 
\end{enumerate}
\end{theorem}
\begin{proof}
It suffices to show that the smallest class of morphisms which satisfies the properties:
\begin{itemize}
\item[(i)]  \label{item:1} it contains $S$ and the weak equivalences in $\M$,
\item[(ii)] \label{item:2} it has the 2-out-of-3 property,
\item[(iii)]\label{item:3} it is closed under homotopy pushouts in $\bf{M}$, 
\item[(iv)] \label{item:4} it is closed in $\bf{M}^{\to}$ under homotopy colimits in $\bf{M}$,
\end{itemize}
and 
\begin{itemize}
\item[(v)] it is closed under homotopy pullbacks in $\M$,
\end{itemize}
is generated by a set of morphisms $\widetilde{S}$ with respect to properties (i)-(iv) only (since these properties specify the classes of weak equivalences of left Bousfield localizations). 
This is proved for $\infty$-topoi in \cite[Prop.~6.2.1.2]{htt}. The proof for model topoi is similar or can easily be obtained indirectly by passing to the associated $\infty$-topos 
and back.
\end{proof}

We emphasize the special dependence of $\widetilde{S}$ on $\M$ that comes from property~(v).
It is easy to conclude that this homotopy left exact Bousfield localization also has the following universal property and therefore may be regarded 
as a kind of ``topofication'' of the pair $(\M, S)$. 

\begin{proposition} \label{topofication-univ-prop}
Let $\M$ be a model topos and $S$ a set of morphisms.
Suppose that $F\colon \M \to \N$ is a left Quillen functor which is homotopy left exact. 
Then $F$ descends to a left Quillen functor on $\M_{\widetilde{S}}$ if and only if it descends to a left Quillen functor on $\M_S$, that is, 
if and only if the left derived functor of $F$ sends $S$ to isomorphisms in $\mathrm{Ho}(\N)$.
In this case, the induced left Quillen functors are again homotopy left exact. 
\end{proposition}
\begin{proof}
Suppose that $F$ descends to a left Quillen functor on $\M_S$. Let $T$ be the class of morphisms in $\M$ which map under $F \circ (-)^c$ to weak equivalences in $\N$.
Here $(-)^c$ denotes a cofibrant replacement functor in $\M$. 
Then the class $T$ satisfies the properties (i)-(v) listed in the proof of Theorem~\ref{thmconstructingmodeltopoi}:
(i) holds by assumption, (ii) is obvious, (iii)-(iv) hold because $F$ is a left Quillen functor, and property (v) is satisfied because $F$ is homotopy left exact.
Thus, $\tilde{S} \subseteq T$ and the result follows. 
\end{proof}

\subsection{Slice categories and restricted homotopical descent}

Let $\M$ be a model topos and $X \in \M$.
By \cite[Cor.~6.10]{Re1}, the slice model category $\M/X$ is again a model topos. 
On the other hand, the slice model category $X/ \M$ is not a model topos in general. (For a quick verification of this claim, simply choose a homotopy pushout with upper left corner $X$, which is not a homotopy pullback, and apply Example~\ref{disjoint-coproducts}.) However, this slice model category still satisfies the homotopical descent properties if we restrict to diagrams over contractible categories. 

\begin{proposition}
Let $\M$ be a model topos and $X$ a cofibrant object in $\M$.
Then the model category $X / \M$ satisfies the homotopical descent properties \eqref{HD1} and \eqref{HD2} of Definition~\ref{descent} for each category 
$I$ whose nerve is weakly contractible.
\end{proposition}
\begin{proof}
We claim that the forgetful functor $U\colon X / \M \to \M$ preserves and detects all homotopy limits and homotopy colimits over contractible categories.
$U$ is right Quillen and it is easy to see that it preserves and detects (homotopy) limits.
Note that $U$ does not preserve colimits in general (but it preserves connected colimits).
Without loss of generality, we may assume that $\M$ is simplicial.
Then the standard model for the homotopy colimit functor gives the following comparison: for a diagram $F\colon I \to X / \M$, there is a homotopy pushout in $\M$
\[
\xymatrix{X \otimes N(I) \ar[r] \ar[d] & X \ar[d] \\
\mathrm{hocolim}_I (UF) \ar[r] & \mathrm{hocolim}_I F.
}
\]
As a consequence, $U$ preserves and detects homotopy colimits when $N(I)$ is weakly contractible.
Then the required result is a direct consequence of the homotopical descent properties of $\M$.
\end{proof}

The following proposition shows that a stable model category automatically fulfills the restricted descent properties of the previous proposition.

\begin{proposition}\label{stablesatisfiesrestricteddescent}
A stable model category $\M$ satisfies the homotopical descent properties \eqref{HD1} and \eqref{HD2} of Definition~\ref{descent} for each category $I$ 
whose nerve is weakly contractible.
\end{proposition}
\begin{proof}
According to the defining property of stable model categories, a commutative square is a homotopy pushout if and only 
if it is a homotopy pullback. 
Suppose that $Y \colon I^{\triangleright} \to \bf{M}$ is a homotopy colimit diagram, $X\colon I^{\triangleright} \to \bf{M}$ is a functor and $\phi\colon X \to Y$ is a natural transformation such that for every $i \to j$ in $I$ the diagram
\[
\xymatrix{
X(i) \ar[r] \ar[d] & X(j) \ar[d] \\
Y(i) \ar[r] & Y(j) 
}  
\]
is a homotopy pushout. Hence, the diagram $Z\colon I \to \bf{M}$ which consists of the (weakly equivalent) vertical homotopy cofibers $\hocofib(X(i)\to Y(i))$ is homotopically constant. First, we note that $X$ is a homotopy colimit diagram 
if and only if the canonical map $\hocolim Z \to \hocofib(X(\infty)\to Y(\infty))$ is a weak equivalence.
Secondly, the diagram
\[
\xymatrix{
X(i) \ar[r] \ar[d] & X(\infty) \ar[d] \\
Y(i) \ar[r] & Y(\infty) 
}  
\]
is a homotopy pushout if and only if $Z(i) \to \hocofib(X(\infty)\to Y(\infty))$ is a weak equivalence.
Hence, it remains to show that $Z(i) \rightarrow \hocolim Z$ is a weak equivalence for all $i \in I$. This follows from 
\cite[Lemma~27.8]{Chacholski-Scherer} given that the nerve of $I$ is weakly contractible.
\end{proof}

\subsection{Right properness and \eqref{HD1}} 

The defining property of a model topos is partially related to the existence of a \emph{right proper} small presentation $\SPre(\C)_S$. 
Right properness is equivalent to the property that for every weak equivalence $f\colon X\to Y$, the Quillen adjunction 
\[
f_! :{\bf{M}}/X \rightleftarrows {\bf{M}}/Y : f^ *,
\]
which is defined by composition with $f$ and pullback respectively, is a Quillen equivalence.
In particular, right properness depends only on the underlying category with weak equivalences.
We emphasize that right properness is \emph{not} invariant under Quillen equivalences (for example, the Bergner model structure on simplicially enriched categories is right proper, whereas the Quillen equivalent Joyal model structure on simplicial sets is not right proper).

\begin{proposition}\label{base-change}
Every model topos admits a right proper small presentation. 
\end{proposition}
\begin{proof}
Let $\bf{M}$ be a model topos and $(\C, S, F)$ a small presentation of $\M$ where $(\C, S)$ is a model site. 
Consider a pullback square in $\SPre(\C)_S$ 
\[
\xymatrix{
X \ar[r]^{g'} \ar[d] & E \ar[d]^{p} \\
A \ar[r]^g & B
}
\]
where $p$ is an $S$-local fibration and $g$ an $S$-local equivalence.
Then $p$ is also a fibration in $\SPre(\C)$.
Since $\SPre(\C)$ is right proper, it follows that the square is also a homotopy pullback in $\SPre(\C)$.
Then it is also a homotopy pullback in $\SPre(\C)_S$ and therefore $g'$ is an $S$-local equivalence, as required.
\end{proof} 

The following partial converse shows that \eqref{HD1} is also a consequence of right properness.
We note that \eqref{HD1} asserts that homotopy colimits commute with homotopy pullbacks and thus can be regarded as a homotopy theoretic analogue of the property that colimits are universal.
We note that \eqref{HD2} does not follow from the existence of a right proper small presentation in general (see, e.g.,~Proposition~\ref{prop:motivicisnotatopos} for an example). 

\begin{theorem}\label{pdescent}
A presentable model category $\bf{M}$ satisfies \eqref{HD1} if and only if it admits a right proper small presentation.
\end{theorem}
\begin{proof}
A direct proof of the ``if''-part can be given along the lines of \cite{Re2}. A complete proof can be found in 
\cite[Prop.~7.8 and Thm.~7.10]{GK}.
\end{proof}

\section{Local model structures} \label{local-model-structures}

\subsection{The $\Uc$-local model structure}
Let $\C$ be a small simplicial category whose underlying ordinary category $\C_0$ is endowed with a Grothendieck topology $\tau$.
For technical convenience, we shall assume that the associated topos of sheaves on $\C_0$ has enough points.
Let $\mathrm{Sh}(\C_0)$ denote the Grothendieck topos of sheaves on (the ordinary site) $\C_0$ and fix a small collection of enough points 
$x_i^* \colon \mathrm{Sh}(\C_0) \to \mathbf{Set}$.
We consider the composite functors 
\[
\hat{x}_i^* \colon \mathrm{PSh}(\C_0) \xrightarrow{\alpha} \mathrm{Sh}(\C_0) \xrightarrow{x_i^*} \mathbf{Set} 
\]
where $\alpha$ denotes the sheafification functor for the $\tau$-topology.
Each functor $\hat{x}_i^*$ induces a functor $\sPre(\C_0) \to \sSet$ which we denote by the same symbol.

We recall that $\Uc\colon \SPre(\C)\to \sPre(\C_0)$ denotes the forgetful functor.
A morphism $\eta\colon F \to G$ in $\SPre(\C)$ is called a \emph{local weak equivalence} if it induces weak equivalences of simplicial sets
\[
(\hat{x}_i^*\Uc)(\eta)\colon (\hat{x}_i^* \Uc)(F) \to (\hat{x}_i^*  \Uc)(G)
\]
for every point $\hat{x}_i^*$.
This class of weak equivalences does not depend on the choice of points $x_i^*$ and it can be equivalently defined in terms of sheaves of homotopy groups (see \cite{Ja}).
An objectwise weak equivalence is also a local weak equivalence \cite[Lemma~9]{Ja}.

A morphism $\eta\colon F \to G$ is a \emph{global fibration} if it has the right lifting property with respect to all morphisms which are projective cofibrations and local weak equivalences.
If $\eta \colon F \to G$ is a global fibration, then it is also an objectwise fibration and $(\hat{x}_i^*\Uc)(\eta)$ is a fibration of simplicial sets for each $\hat{x}_i^*$.
This follows from the fact that $\hat{x}_i^* \Uc$ preserve finite limits and epimorphisms.
If $\C$ is an ordinary site, the corresponding notion of a globally fibrant object essentially encodes the property of being a homotopy sheaf (with respect to $\tau$-hypercovers).
We refer to \cite{DHI} and \cite{Ja} for background on homotopical sheaf theory in the case where $\C$ is an ordinary (non-simplicial) category. 

\begin{theorem} \label{local-simp}
Let $\C$ be a small simplicial category whose underlying ordinary category $\C_0$ is endowed with a Grothendieck topology $\tau$.
Then the classes of projective cofibrations, local weak equivalences and global fibrations define a proper simplicial combinatorial model structure on the category $\SPre(\C)$. 
\end{theorem}
\begin{proof}
We show that the conditions of Smith's recognition theorem for model structures on locally presentable categories are satisfied (see \cite[Prop.~A.2.6.10]{htt}, \cite[Thm.~4.1]{Ra}). 

The class of local weak equivalences is the intersection of the preimages of the class of weak equivalences between simplicial sets along the small collection of accessible functors $\hat{x}_i^* \Uc$ for each point $\hat{x}_i^*$.
The class of weak equivalences between simplicial sets, regaded as a full subcategory of $\sSet^{\to}$, is accessible and acccessibly embedded \cite[Cor.~A.2.6.8]{htt}, \cite{Ra2}.
It follows that the class of local weak equivalences is accessible and accessibly embedded in $\SPre(\C)^{\to}$, regarded as a full subcategory.
It also has the 2-out-of-3 property. 

A morphism which has the right lifting property with respect to the projective cofibrations is an objectwise weak equivalence and therefore also a local weak equivalence.
Lastly, the class of local weak equivalences which are monomorphisms is cofibrantly closed (that is, it is closed under pushouts, transfinite compositions and retracts), since the functors of points $\hat{x}_i^*\Uc$ preserve colimits, monomorphisms and weak equivalences, and the corresponding property is valid in $\sSet$.
Hence the intersection of projective cofibrations and local weak equivalences is also cofibrantly closed.
This completes the proof of the existence of the model structure.

The compatibility with the simplicial structure and left properness follow easily from Theorem~\ref{proj-simp}.
Right properness follows from the right properness of $\sSet$ given that the functors $\hat{x}_i^*\Uc$ preserve pullbacks and send global fibrations to fibrations of simplicial sets. 
\end{proof}

This model category will be denoted by $\SPre(\C)_{\mathrm{\mathcal{U}\tau}}$.
We will refer to it as the \emph{$\mathcal{U}$-local model structure} on $\SPre(\C)$ in order to emphasize that the simplicial structure and the Grothendieck topology are given independently of each other.
We note that it is a left Bousfield localization of the projective model category $\SPre(\C)$ at the class of local weak equivalences.

In the case of Theorem~\ref{local-simp} where $\C$ is an ordinary category, we will usually denote the model category $\SPre(\C)_{\mathrm{\mathcal{U}\tau}}$ by $\sPre(\C, \tau)$ and refer to it as the \emph{local model structure} (see \cite{Blander, Ja}). 

\begin{remark} \label{choices-of-cof}
As the proof of Theorem~\ref{local-simp} suggests, it is also possible to choose larger classes of cofibrations.
Any set of monorphisms which contains the generating projective cofibrations generates a class of cofibrations for a model structure 
on $\SPre(\C)$ where the weak equivalences are the local weak equivalences. 
\end{remark}

We show next that the $\Uc$-local model structures are model topoi.
This is well known in the case of ordinary Grothendieck sites (see~\cite{Re2}). 

\begin{theorem} \label{local-topos}
Let $\C$ be a small simplicial category whose underlying ordinary category $\C_0$ is endowed with a Grothendieck topology $\tau$.
Then the $\Uc$-local model category $\SPre(\C)_{\mathrm{\mathcal{U}\tau}}$ is a model topos.
\end{theorem}
\begin{proof}
By Proposition~\ref{criterion}, it suffices to show that for every pullback square in $\SPre(\C)$ 
\[
\xymatrix{
X \ar[r]^{g'} \ar[d] & Y \ar[d]^p \\
X' \ar[r]^g & Y'
}
\]
where $p$ is a objectwise fibration and $g$ is a local weak equivalence, then $g'$ is also a local weak equivalence.
This is a consequence of the right properness of $\sSet$ using the fact that the functors of points $\hat{x}_i^*$ preserve pullbacks and send objectwise fibrations to fibrations of simplicial sets.
\end{proof}

\begin{remark}(Naturality)\label{remark:naturality}
Let $\C$ and $\C'$ be small simplicial categories whose underlying categories $\C_0$ and $\C_0'$ are equipped with Grothendieck topologies $\tau$ and $\tau'$.
Let $F \colon \C \to \C'$ be a simplicial functor which restricts to a morphism 
of sites $F_0\colon (\C_0,\tau)\to(\C'_0,\tau')$.
There is a Quillen adjunction between projective model categories
\[
F_{!} : \SPre(\C) \rightleftarrows \SPre(\C') : F^* .
\]
However, the functor $\SPre(\C)_{\Uc \tau} \xrightarrow{F_!} \SPre(\C')_{\Uc \tau'}$ is not a left Quillen functor in general.
To see this, let $\C$ be a simplicial category with underlying category $\C_0$ considered as a discrete simplicial category.
There is a canonical simplicial functor 
\[
F\colon \C_0 \to \C
\]
which is the identity on objects.
The associated adjunction $(F_!, F^*)$ can be identified with the adjunction $(\Hs, \Uc)$.
But the adjunction
\[
\Hs: \sPre(\C_0, \tau) \rightleftarrows \SPre(\C)_{\Uc\tau}: \Uc
\]
is \emph{not} a Quillen adjunction in general (see Corollary~\ref{notanadjunction}).
\end{remark}

\begin{remark} \label{U-left-Quillen}
The functor $\Uc \colon \SPre(\C)_{\Uc \tau} \to \sPre^{\mathrm{inj}}(\C_0, \tau)$ is a \emph{left} Quillen functor if we use the local injective model category $\sPre^{\mathrm{inj}}(\C_0, \tau)$ where the cofibrations are the monomorphisms and the weak equivalences are the local weak equivalences defined as before. $\Uc$ has a right adjoint and it preserves monomorphisms and weak equivalences.
Moreover, $\Uc$ preserves and detects homotopy pullbacks. To see this, it suffices to note that homotopy pullbacks in these model categories can be calculated by replacing morphisms by local fibrations, that is, morphisms which restrict to fibrations of simplicial sets at every point 
$x_i^*$ of $\mathrm{Sh}(\C_0)$.
\end{remark}

\subsection{Model topoi from Grothendieck topologies on $\mathrm{Ho}(\C)$} \label{sheaves}
General constructions of model topoi (or $\infty$-topoi) that arise from a Grothendieck topology were introduced and studied in \cite{TV} and \cite{htt}.
In that context, a Grothendieck topology on a simplicial category (or $\infty$-category) $\C$ is a Grothendieck topology on the associated homotopy category $\mathrm{Ho}(\C)$.
This context differs from our main example of a model topos, the $\Uc$-local model topos (see Theorem~\ref{local-topos}), because there the Grothendieck topology and the simplicial enrichment are given 
independently. The purpose of this subsection is to review some parts of the theory of model topoi from \cite{TV} before we discuss the connection with the $\Uc$-local model topoi in the 
next subsection. 

Let $\C$ be a small simplicial category with a Grothendieck topology $\ttau$ on $\mathrm{Ho}(\C)$.
For each simplicial presheaf $F \in \SPre(\C)$, there is an associated sheaf of connected components $\tilde\pi_0(F)$ on $\mathrm{Ho}(\C)$ and sheaves of homotopy groups $\tilde\pi_n(F,s)$ on $\mathrm{Ho}(\C/x)$, for $n \geq 1$ and $s \in \tilde\pi_0(F(x))$.
(These are denoted $\pi_0(F)$ and $\pi_n(F, s)$, respectively, in \cite{TV}.)
These are the $\ttau$-sheaves associated to taking homotopy groups objectwise. A morphism $\eta \colon F \to G$ in $\SPre(\C)$ is a \emph{$\tilde \pi_*$-equivalence} if it induces isomorphisms of sheaves 
\begin{align*}
 \tilde{\pi}_0(F)    &\to \tilde{\pi}_0(G)\\
 \tilde{\pi}_n(F, s) &\to \tilde{\pi}_n(G, \eta(s))
\end{align*}
for all $n\geq 1$ and sections $s \in \tilde\pi_0(F(x))$ (see \cite[Sect.~3]{TV}).
We say that $\eta\colon F \to G$ is a \emph{global fibration} if it has the right lifting property with respect to all morphisms which are projective cofibrations and $\tilde \pi_*$-equivalences.
The corresponding notion of a globally fibrant object encodes the property of being a homotopy sheaf with respect to hypercovers defined by $\ttau$ (see \cite[3.4]{TV}). 

\begin{theorem}[To\"en--Vezzosi \cite{TV}] \label{toen-vezzosi}
Let $\C$ be a small simplicial category with a Grothendieck topology $\ttau$ on $\mathrm{Ho}(\C)$.
Then the classes of projective cofibrations, \mbox{$\tilde\pi_*$-equivalences} and global fibrations define a proper simplicial combinatorial model structure on the category $\SPre(\C)$. 
\end{theorem}
\begin{proof}
See \cite[Thm.~3.4.1]{TV}.
\end{proof}

We denote this model category by $\SPre(\C, \ttau)$.
The left Quillen functor 
\[
 \mathrm{id} \colon \SPre(\C) \to \SPre(\C, \ttau)
\]
is homotopy left exact and therefore $\SPre(\C, \ttau)$ is a model topos \cite[Prop.~3.4.10]{TV}.

\begin{remark}
Let $\C$ be an ordinary category, considered as a discrete simplicial category, and let $\tau=\ttau$ be a Grothendieck topology on $\C=\mathrm{Ho}(\C)$.
In this case, the model structure $\SPre(\C, \ttau)$ from Theorem~\ref{toen-vezzosi} agrees with the $\Uc$-local model structure $\SPre(\C)_{\mathrm{\mathcal{U}\tau}}$ 
from Theorem~\ref{local-simp} and both agree with the local model structure on $\sPre(\C)$.
In particular, there is no conflict with the notation $\sPre(\C,\tau)$ introduced before Remark~\ref{choices-of-cof}.
\end{remark}

Moreover, we have the following classification theorem. 

\begin{theorem}[To\"en--Vezzosi \cite{TV}] \label{toen-vezzosi2}
Let $\C$ be a small simplicial category.
Then there is a bijective correspondence between Grothendieck topologies $\ttau$ on $\mathrm{Ho}(\C)$ and homotopy left exact left Bousfield localizations of $\SPre(\C)$ which are $t$-complete. 
\end{theorem}
\begin{proof}
See \cite[Thm.~3.8.3]{TV}.
\end{proof}

The notion of $t$-completeness (or hypercompleteness \cite{htt}) refers to hyperdescent as opposed to plain descent with respect to the \v{C}ech covers.
In other words, it means that the class of weak equivalences can be specified in terms of homotopy sheaves or, equivalently, that it can be detected by truncated objects.
We refer to \cite{TV, htt} for more details. 

\begin{remark}
The $\infty$-topoi of sheaves in \cite{htt} are defined in terms of \v{C}ech descent, that is, they are obtained as localizations of $\infty$-categories of presheaves at the collection of covering sieves that define the Grothendieck topology.
These $\infty$-topoi define \emph{topological} localizations \cite[Def.~6.2.1.4, Prop.~6.2.2.7]{htt}.
Lurie \cite{htt} proved a related classification result saying that there is a bijective correspondence between Grothendieck topologies $\ttau$ on $\mathrm{Ho}(\C)$ and topological localizations of the presentable $\infty$-category of presheaves associated to $\C$ \cite[Prop.~6.2.2.17]{htt}.
The model topos of Theorem~\ref{toen-vezzosi} corresponds to the hypercompletion (or $t$-completion) of the $\infty$-topos of sheaves in the sense of Lurie \cite{htt}.
\end{remark}

\begin{remark}\label{bijectivecorrespondenceofTV}
We recall the definition of the bijective correspondence in Theorem~\ref{toen-vezzosi2}.
One direction is given by the construction of Theorem~\ref{toen-vezzosi}.
For the other direction, consider a homotopy left exact left Bousfield localization 
\[
\mathrm{id}_S \colon \SPre(\C) \to \SPre(\C)_S
\]
from which we want to extract a Grothendieck topology $\ttau$ on $\mathrm{Ho}(\C)$.
The adjunction $(\pi_0\dashv \text{discrete})$ of functors between simplicial sets and sets gives rise to a natural simplicial functor
\[
\C\xrightarrow{\eta} \mathrm{Ho}(\C)
\]
and hence to an adjunction  
\[
\eta_{!} : \SPre(\C) \rightleftarrows \SPre(\mathrm{Ho}(\C)) \cong \sPre(\mathrm{Ho}(\C)) : \eta^* 
\]
between the categories of simplicially enriched presheaves on the respective simplicial categories.
Consider the full sub\-ca\-tego\-ry $\mathrm{PSh}(\mathrm{Ho}(\C))\subseteq \sPre(\mathrm{Ho}(\C))$ of set-valued presheaves.
Then, a sieve on $X \in \mathrm{Ho}(\C)$, 
\[
U \rightarrowtail y_{\mathrm{Ho}(\C)}(X),
\]
is a $\ttau$-covering sieve if
\[
\eta^*(U) \rightarrowtail \eta^*(y_{\mathrm{Ho}(\C)}(X))
\]
is an $S$-local equivalence in $\SPre(\C)$. Here $y$ denotes the Yoneda embedding.    
\end{remark}

\subsection{Comparing Grothendieck topologies}\label{comparinggrothendiecktopologies}

Let $\C$ be a small simplicial category.
The purpose of this subsection is to compare Grothendieck topologies $\tau$ on the underlying category $\C_0$ of $\C$ with Grothendieck topologies $\ttau$ on $\mathrm{Ho}(\C)$, 
as considered by To\"en--Vezzosi~\cite{TV} and Lurie~\cite{htt}, with a view towards comparing the $\Uc$-local model topos $\SPre(\C)_{\mathrm{\mathcal{U}\tau}}$ of Theorem~\ref{local-topos} with the To\"en--Vezzosi model topos $\SPre(\C, \ttau)$ of Theorem~\ref{toen-vezzosi}.
These two constructions of model topoi differ in general because in the first case the definition of the covering sieves does not take into account the simplicial enrichment.

First, using the bijective correspondence from Theorem \ref{toen-vezzosi2}, we can identify the Grothendieck topology $\ttau$ on $\mathrm{Ho}(\C)$ that is associated with the $\Uc$-local model topos.
Let $\C$ be a small simplicial category whose underlying ordinary category $\C_0$ is endowed with a Grothendieck topology $\tau$.
The triple of functors $(\pi_0\dashv \text{discrete}\dashv -_0)$ between simplicial sets and sets induces two natural simplicial functors
\[
\C_0 \xrightarrow{\epsilon} \C\xrightarrow{\eta} \mathrm{Ho}(\C)
\]
whose composition is the localization functor $\gamma \colon \C_0 \to \mathrm{Ho}(\C)$.
We obtain two simplicially enriched adjunctions of the associated presheaf categories
\begin{equation}\label{compositeadjunction}
  \sPre(\C_0) \cong \SPre(\C_0) \mathrel{\mathop{\myrightleftarrows{\rule{3ex}{0ex}}}^{\epsilon_!}_{\epsilon^*}} \SPre(\C) \mathrel{\mathop{\myrightleftarrows{\rule{3ex}{0ex}}}^{\eta_!}_{\eta^*}} \SPre(\mathrm{Ho}(\C)) \cong \sPre(\mathrm{Ho}(\C)) 
\end{equation}
with composite adjunction
\[
\gamma_! : \sPre(\C_0) \rightleftarrows \sPre(\mathrm{Ho}(\C)) : \gamma^* 
\]
where the right adjoints are given by precomposition with (the opposite of) the respective functor.
We use the same notation to denote the restriction of this last adjunction to the set-valued presheaf categories
$$
\gamma_! : \mathrm{PSh}(\C_0) \rightleftarrows \mathrm{PSh}(\mathrm{Ho}(\C)) : \gamma^*  
$$
Note that the adjunction $(\epsilon_!,\epsilon^*)$ is identified with the adjunction $(\Hs,\Uc)$ from \eqref{eq:adjunctionHandU}, and that the adjunction $(\eta_!,\eta^*)$ was already considered in Remark~\ref{bijectivecorrespondenceofTV}.

Following the description of the bijection in Theorem~\ref{toen-vezzosi2} as explained in Re\-mark~\ref{bijectivecorrespondenceofTV}, we say that a sieve on $X \in \mathrm{Ho}(\C)$, $U \rightarrowtail y_{\mathrm{Ho}(\C)}(X)$, is a $[\tau]$-\emph{covering sieve} if 
\[
\eta^*(U) \rightarrowtail \eta^*(y_{\mathrm{Ho}(\C)}(X))
\]
is a $\Uc$-local equivalence in $\SPre(\C)$. Let $[\tau]$ denote the collection of \mbox{$[\tau]$-covering} sieves. 
By Theorem~\ref{local-topos} and using similar arguments as in the definition of the bijection in Theorem~\ref{toen-vezzosi2}, it follows that $[\tau]$ defines a Grothendieck 
topology. Indeed the left exact Bousfield localization from Theorem~\ref{local-topos} induces a left exact localization of the category of presheaves on $\mathrm{Ho}(\C)$ after 
restricting to the $0$-truncated objects. By definition, this left exact localization corresponds to the Grothendieck topology $[\tau]$ (see also \cite{TV}). 

We write $\alpha_{\tau}$ for the $\tau$-sheafification functor on $\mathrm{PSh}(\C_0)$ and call a morphism in $\mathrm{PSh}(\C_0)$ a \emph{$\tau$-isomorphism} if it becomes an isomorphism after \mbox{$\tau$-sheafification}.
Likewise, we write $\alpha_{[\tau]}$ to denote the $[\tau]$-sheafification functor on $\mathrm{PSh}(\mathrm{Ho}(\C))$ and say that a morphism in $\mathrm{PSh}(\mathrm{Ho}(\C))$ is 
a \emph{$[\tau]$-isomorphism} if it becomes an isomorphism after $[\tau]$-sheafification.

\begin{remark}\label{gammaremark}
Using the composite adjunction \eqref{compositeadjunction} and the identification $\Uc\cong \epsilon^*$, a sieve $U \rightarrowtail y_{\mathrm{Ho}(\C)}(X)$ is a $[\tau]$-\emph{covering sieve} if and only if $\gamma^*(U) \rightarrowtail \gamma^*(y_{\mathrm{Ho}(\C)}(X))$
is a $\tau$-isomorphism.
\end{remark}

\begin{lemma} \label{compare-sheafifications}
Let $\C$ be a small simplicial category whose underlying ordinary category $\C_0$ is endowed with a Grothendieck topology $\tau$.
Then a morphism $f \colon F \to G$ in $\mathrm{PSh}(\mathrm{Ho}(\C))$ is a $[\tau]$-isomorphism if and only if the morphism in $\mathrm{PSh}(\C_0)$
\[
\gamma^{*}(f) \colon \gamma^*(F)  \to \gamma^*(G) 
\]
is a $\tau$-isomorphism.
\end{lemma}
\begin{proof}
This follows from unwinding the definitions. The $[\tau]$-sheafification functor $\alpha_{[\tau]} \colon \mathrm{PSh}(\mathrm{Ho}(\C)) \to \mathrm{Sh}(\mathrm{Ho}(\C), [\tau])$ is identified by definition with the restriction of the homotopy left exact left Bousfield localization $\SPre(\C) \to \SPre(\C)_{\Uc \tau}$ to the 0-truncated objects. Thus, $f \colon F \to G$ is a $[\tau]$-isomorphism if and only if $\eta^*(f) \colon \eta^*(F) \to \eta^*(G)$ is a weak equivalence in $\SPre(\C)_{\Uc \tau}$, that is, if and only if the morphism $\gamma^*(f) \colon \gamma^*(F) \to \gamma^*(G)$ is a local weak equivalence 
in $\sPre(\C_0)$, which means that $\gamma^*(f)$ is a $\tau$-isomorphism.    
\end{proof}

\begin{remark} \label{cover-reflecting}
Using \cite[Prop.~C2.3.18]{JoTopos}, the previous Lemma~\ref{compare-sheafifications} implies that the Grothendieck topology $[\tau]$ makes the functor $\gamma \colon \C_0 \to \mathrm{Ho}(\C)$ \emph{cover-reflecting}.
This means that given a $[\tau]$-covering sieve $U \rightarrowtail y_{\mathrm{Ho}(\C)}(X)$, then the sieve on $X \in \C$ which consists of all $f \colon V \to X$ in $\C$ such that $[f] \in U$ is a $\tau$-covering sieve.
Moreover, the right Kan extension 
\[
\gamma_*\colon\mathrm{PSh}(\C_0) \to \mathrm{PSh}(\mathrm{Ho}(\C)) 
\]
sends $\tau$-sheaves to $[\tau]$-sheaves (see \cite[Prop.~C2.3.18]{JoTopos}, \cite[III.2]{SGA4}).  
\end{remark}

\begin{proposition} \label{tauDelta}
Let $\C$ be a small simplicial category whose underlying ordinary category $\C_0$ is endowed with a Grothendieck topology $\tau$.
Then the $\Uc$-local model category $\SPre(\C)_{\Uc \tau}$ is the same as $\SPre(\C, [\tau])$. 
\end{proposition}
\begin{proof}
We recall that a morphism $f \colon F \to G$ in $\SPre(\C)$ is a $\Uc$-local weak equivalence if $\Uc(f)$ is a $\tilde{\pi}_*$-equivalence in $\sPre(\C_0, \tau)$.
We need to compare this class of morphisms with the class of $\tilde{\pi}_*$-equivalences in $\SPre(\C, [\tau])$. 
For our purposes here, it will be more convenient to use the characterization of \mbox{$\tilde{\pi}_*$-equivalences} in $\SPre(\C, [\tau])$ which does not involve basepoints \cite[Lemma~3.3.3]{TV}. 
According to this, an objectwise fibration $F \to G$ between objectwise fibrant objects in $\SPre(\C, [\tau])$ is a $\tilde{\pi}_*$-equivalence if for any $n \geq 0$, the induced morphism 
\[
F^{\Delta^n} \to F^{\partial \Delta^n} \times_{G^{\partial \Delta^n}} G^{\Delta^n} 
\]
is a $\tilde{\pi}_0$-isomorphism (with respect to $[\tau]$).
Note that there is a similar characterization of the weak equivalences in $\sPre(\C_0, \tau)$.
Then it follows from Lemma~\ref{compare-sheafifications} that a morphism $F \to G$ in $\SPre(\C)$ is a $\tilde{\pi}_*$-equivalence in $\SPre(\C, [\tau])$ if and only if it is a $\Uc$-local weak equivalence. The result follows. 
\end{proof}

The Grothendieck topology $[\tau]$ on $\mathrm{Ho}(\C)$ admits a more explicit description as follows.
Given a $\tau$-covering sieve $J \colon U \rightarrowtail y_{\C_0}(X)$ on $X \in \C_0$ which is generated by 
$\{f_{\alpha} \colon X_{\alpha} \to X\}$, let 
$$[J]\colon [U] \rightarrowtail y_{\mathrm{Ho}(\C)}(X)$$
denote the sieve on $X \in \mathrm{Ho}(\C)$ which is generated by $\{\gamma(f_{\alpha}) \colon X_{\alpha} \to X\}$.  

\begin{lemma} \label{identification-topologies}
Let $\C$ be a small simplicial category whose underlying ordinary category $\C_0$ is endowed with a Grothendieck topology $\tau$. 
A sieve $j \colon U \rightarrowtail y_{\mathrm{Ho}(\C)}(X)$ is a $[\tau]$-covering sieve if and only if it is of the form $[J]$ for some $\tau$-covering sieve 
$J \colon \widetilde{U} \rightarrowtail y_{\C_0}(X)$. 
\end{lemma}
\begin{proof}
Suppose that $j$ is a $[\tau]$-covering sieve.
Consider the pullback of presheaves on $\C_0$,
\[
 \xymatrix{
 \widetilde{U} \ar@{>->}[d]^{J} \ar@{->>}[r] & \gamma^*(U) \ar@{>->}[d]^{\gamma^*(j)} \\
 y_{\C_0}(X) \ar@{->>}[r] & \gamma^*(y_{\mathrm{Ho}(\C)}(X))
 }
\]
and apply $\tau$-sheafification to obtain a new pullback square 
\[
 \xymatrix{
  \alpha_{\tau}(\widetilde{U}) \ar@{>->}[d]^{\cong} \ar@{->>}[r] & \alpha_{\tau}\gamma^*(U) \ar[d]^{\cong} \\
 \alpha_{\tau}(y_{\C_0}(X)) \ar@{->>}[r] & \alpha_{\tau}\gamma^*(y_{\mathrm{Ho}(\C)}(X))
 }
\]
whence it follows that $J$ is a $\tau$-covering sieve (cf. Remark~\ref{cover-reflecting}).
Note that the composite morphism $ \widetilde{U} \to \gamma^*(y_{\mathrm{Ho}(\C)}(X))$ factors as follows 
\[
\widetilde{U} \twoheadrightarrow \gamma^*[\widetilde{U}] \rightarrowtail \gamma^*(y_{\mathrm{Ho}(\C)}(X)) 
\]
where the first morphism is an epimorphism.
Comparing with the factorization in the first diagram above, it follows that $[J] = j$. 

For the converse, suppose that $J \colon U \rightarrowtail y_{\C_0}(X)$ is a $\tau$-covering sieve.
Consider the pullback $U_{\Delta}$ of the following presheaves on $\C_0$,
\bigskip
\[
 \xymatrix{
 U \ar@{>->}[rd]_{J} \ar@{>-->}[r] \ar@/^1.7pc/@{->>}[rr] & U_{\Delta} \ar@{>->}[d]^{J_{\Delta}}  \ar@{->>}[r] & \gamma^*([U]) \ar@{>->}[d]^{\gamma^*[J]} \\
 & y_{\C_0}(X) \ar@{->>}[r] & \gamma^*(y_{\mathrm{Ho}(\C)}(X))
 }
\]
The sieve $J_{\Delta}$ is again a $\tau$-covering sieve since it contains $J$.
Applying $\tau$-sheafification $\alpha_\tau$, we obtain a pullback as follows
 \[
 \xymatrix{
 \alpha_{\tau}(U_{\Delta}) \ar@{>->}[d]^{\cong}  \ar@{->>}[r] & \alpha_{\tau} \gamma^*(([U]) \ar@{>->}[d]^{\alpha_{\tau}\gamma^*[J]} \\
  \alpha_{\tau}(y_{\C_0}(X)) \ar@{->>}[r] & \alpha_{\tau}\gamma^*(y_{\mathrm{Ho}(\C)}(X)) 
 }
\]
So $\alpha_{\tau}(\gamma^*[J])$ is an isomorphism and therefore $\gamma^*[J]$ is a $\tau$-covering sieve, as required. 
\end{proof}

The correspondence $J \mapsto J_{\Delta}$ that appears in the proof of Lemma \ref{identification-topologies} can be used to elucidate the main difference between the topologies $\tau$
and $[\tau]$. This correspondence sends a covering sieve $J$ to a larger covering sieve which consists of all elements which are homotopic to an element in $J$. 
It may be considered as a kind of homotopical thickening of $J$. Note that $[J] = [J_{\Delta}]$ and every $[\tau]$-covering sieve is $[J_{\Delta}]$ for a \emph{unique} 
covering sieve of the 
form $J_{\Delta}$. In particular, $[\tau]$ depends only on the homotopical thickenings of $\tau$-covering sieves, i.e., the covering sieves of the form $J_{\Delta}$. Moreover, the Grothendieck topology generated by the sieves 
of the form $J_{\Delta}$, for a $\tau$-covering sieve $J$, is the unique smallest Grothendieck topology on $\C_0$ such that $\gamma \colon \C_0 \to \mathrm{Ho}(\C)$ 
is cover-reflecting (see \cite[Lemma~C2.3.19]{JoTopos}). 

Furthermore, Lemma~\ref{identification-topologies} shows that $[\tau]$ is the smallest Grothendieck topology such that the localization functor $\gamma \colon \C_0 \to \mathrm{Ho}(\C)$ preserves covering sieves (see \cite[Lemma~C2.3.12]{JoTopos}).
But $\gamma$ is not a morphism of sites in general because it fails to satisfy the necessary flatness conditions (see, e.g., \cite[Rem.~C2.3.7]{JoTopos}).
We have the following results about the comparison between the different sheaf conditions. 

\begin{proposition} \label{reflecting-sheaf-condition}
Let $\C$ be a small simplicial category whose underlying ordinary category $\C_0$ is endowed with a Grothendieck topology $\tau$.
Let $F$ be an object of $\mathrm{PSh}(\mathrm{Ho}(\C))$.
If $\gamma^*(F)$ is a $\tau$-sheaf (resp. $\tau$-separated presheaf) on $\C_0$, then the presheaf $F$ is a $[\tau]$-sheaf (resp. $[\tau]$-separated presheaf). Conversely, if $F$ is a $[\tau]$-separated presheaf, then $\gamma^* F$ is a $\tau$-separated presheaf.
\end{proposition}
\begin{proof}
Let $[J]\colon [U] \rightarrowtail y_{\mathrm{Ho}(\C)}(X)$ be a $[\tau]$-covering sieve on $X \in \mathrm{Ho}(\C)$.
We need to show that the top map in the diagram 
\[
\xymatrix{
F(X) \cong \hom(y_{\mathrm{Ho}(\C)}(X), F) \ar[r] \ar@{=}[d] & \hom([U], F) \ar[d] \\
(\gamma^*F)(X) \cong \hom(y_{\C_0}(X), \gamma^*F) \ar[r]^(.65){\cong} & \hom(U, \gamma^*F)
}
\]
is an isomorphism (resp. monomorphism). The vertical maps are induced by $\gamma^*$ and the morphisms $y_{\C_0}(X) \to \gamma^* y_{\mathrm{Ho}(\C)}(X)$ and $U \to \gamma^* [U]$, respectively. The bottom map is an isomorphism (resp. monomophism) because $\gamma^*F$ is a $\tau$-sheaf (resp. $\tau$-separated presheaf). Therefore
the top map is a monomorphism. Since $\gamma^{*}$ is fully faithful and $U \to \gamma^* [U]$ is an epimorphism, the right vertical map is injective and the result follows. Conversely, if the top map is a monomorphism, then so is the bottom map as well. 
\end{proof}

\begin{remark} \label{difference-of-sheaves}
The converse statement for the sheaf condition is false in general, that is, $\gamma^{*} \colon \mathrm{PSh}(\mathrm{Ho}(\C)) \to \mathrm{PSh}(\C)$ does not preserve sheaves in general 
(see Example \ref{example-sheaf-condition} below).  Given a presheaf $F$ on $\mathrm{Ho}(\C)$, then $\gamma^*F$ is a $\tau$-sheaf if and only if $F$ is orthogonal with respect to the set 
of morphisms $\gamma_!(\tau)$ where 
\[
\gamma_! \colon \mathrm{PSh}(\C_0) \rightarrow \mathrm{PSh}(\mathrm{Ho}(\C)) 
\]
is the left adjoint of $\gamma^*$ (see also \cite[III.1]{SGA4}). But note that for a 
$\tau$-covering sieve $J \colon U \rightarrowtail y_{\C_0}(X)$, the induced epimorphism 
\[
q \colon \gamma_!(U) \twoheadrightarrow [U]
\]
is not a monomorphism in general. In general, the $[\tau]$-sheaf condition, i.e., orthogonality with respect to $[\tau]$, 
is \emph{weaker} than the $\tau$-sheaf condition. 
\end{remark}

\begin{remark}
We note the following immediate consequence of Lemma \ref{compare-sheafifications} and Proposition \ref{reflecting-sheaf-condition}. If $F$ is a presheaf on $\mathrm{Ho}(\C)$, 
then $\gamma^*(\alpha_{[\tau]}(F))$ is a $\tau$-sheaf if and only if it is the $\tau$-sheafification of $\gamma^*(F)$. 
\end{remark}

\begin{example} \label{example-sheaf-condition}
Let $\C$ be a simplicial category with only two objects $x$ and $y$ and non-identity morphisms only from $x$ to $y$. Suppose that $\tau$ is the Grothendieck topology on $\C_0$ which 
is given by the sieve generated by \emph{all} the morphisms $\{f_{\alpha} \colon x \to y\}$. If the simplicial set $\mathrm{map}_{\C}(x, y)$ is connected, then $\mathrm{Ho}(\C)$ is equivalent to $[1] = \{0 < 1\}$, regarded as a category. In this case, a presheaf $F \colon \mathrm{Ho}(\C)^{\op} \to \mathrm{Set}$ 
is a $[\tau]$-sheaf if and only if the restriction map $F(y) \to F(x)$ is an isomorphism. On the other hand, a constant presheaf $F \colon \C_0^{\op} \to \mathrm{Set}$ is not a $\tau$-sheaf in general. 
\end{example}

\section{Motivic spaces} \label{motivic-spaces}

\subsection{The enriched category $\SSm_S$}\label{motivicspaces1}

Let $S$ be a noetherian scheme of finite Krull dimension.
Let $\Sm_S$ be the category of smooth schemes of finite type over $S$.
The category $\Sm_S$ is essentially small and we implicitly fix a small skeleton.

Consider the cosimplicial object $\ADelta^\upperminus\colon\Delta\to\Sm_S$ defined by
\[
\textstyle
	\ADelta^n=\Spec S[X_0,\ldots,X_n]/{\left(1-\sum X_i\right)}
\]
and the usual coface and codegeneracy maps.
This defines the structure of a simplicial category $\SSm_S$ on $\Sm_S$ by \cite[Lemma~1.1]{HS} where
\[
\mathrm{map}_{\Sm_S}(A,B)_n=\hom_{\Sm_S}(A\times \ADelta^n, B).
\]
It was observed in \cite[Lemma~1.4]{HS} that the unit of the Quillen adjunction
\begin{equation}\label{adjunction}
 \Hs: \sPre(\Sm_S) \rightleftarrows \SPre(\SSm_S): \Uc 
\end{equation}
is given by the $\Sing$-construction of \cite{MV99}, and we have 
\[
\Uc\Hs(F)(U)_n=\Sing(F)(U)_n=F(U\times\ADelta^n)_n 
\]
for $F\in \sPre(\Sm_S)$ and $U\in\Sm_S$.

We emphasize that every enriched simplicial presheaf $F\in\SPre(\SSm_S)$ is \mbox{$\A^1$-(homotopy) invariant}, i.e.,
\[
F(U)\xrightarrow{\pr^*} F(U\times\A^1)
\]
is a weak equivalence of simplicial sets for every $U\in\SSm_S$ \cite[Lemma~2.8]{HS}.

Consider the set of morphisms $\{U\times\A^1\xrightarrow{\pr} U\mid U\in\Sm_S\}$ and let $\sPre(\Sm_S)_{\A^1}$ be the left Bousfield localization of the projective model category $\sPre(\Sm_S)$ of Theorem~\ref{proj-simp} at this set of 
morphisms. An 
object $F \in \sPre(\Sm_S)_{\A^1}$ is fibrant if and only if it is objectwise fibrant and $\A^1$-invariant. 

\begin{proposition}\label{A1project}
There is a Quillen equivalence
\[
\Hs:\sPre(\Sm_S)_{\A^1}\rightleftarrows \SPre(\SSm_S):\Uc.
\]
The right adjoint $\Uc$ detects weak equivalences and fibrations.
\end{proposition}
\begin{proof}
The adjunction $(\Hs, \Uc)$ is a simplicial adjunction by construction and the respective model structures in the Proposition are both simplicial and left proper.
To see that it defines a Quillen adjunction, it suffices to show that $\Hs$ preserves cofibrations and $\Uc$ preserves fibrant objects (see \cite[Cor.~A.3.7.2]{htt}). The first is clear as left Bousfield localizations do not change the cofibrations and $(\Hs, \Uc)$ is a Quillen adjunction between the projective model structures. The right adjoint $\Uc$ sends fibrant objects to $\A^1$-invariant fibrant objects. Hence there is an induced Quillen adjunction as claimed. 

This Quillen adjunction is also a Quillen equivalence because the canonical map $A\to \Sing(A)$ is an $\A^1$-equivalence as implied by \cite[Cor.~2.3.8]{MV99}.
\end{proof}

\begin{remark}
As a consequence of the last proposition, there is an equivalence between the homotopy category of $\SPre(\SSm_S)$ and the full subcategory of the homotopy category of $\sPre(\Sm_S)$ consisting of objects of the form $\Sing(X)$ for some $X \in \sPre(\Sm_S)$.
In particular, this means that a natural transformation between two such simplicial presheaves is equivalent to a simplicially enriched one, uniquely up to homotopy.
This observation extends to show also a weak equivalence (i.e., DK-equivalence) between the associated simplicial categories of fibrant-cofibrant objects.
\end{remark}

\begin{remark}
Since $\SPre(\SSm_S)$ is a model topos so is $\sPre(\Sm_S)_{\A^1}$, too.
However, the left Bousfield localization 
\begin{equation}\label{A1localization}
 \id_{\A^1} \colon \sPre(\Sm_S) \to \sPre(\Sm_S)_{\A^1}
\end{equation}
is not homotopy left exact.
This can be seen as a consequence of the fact that the motivic homotopy theory is not a model topos (see Proposition~\ref{prop:motivicisnotatopos} below).
Note that a fibrant replacement functor for $\sPre(\Sm_S)_{\A^1}$ is given by the $\Sing$-functor (post-composed with an objectwise fibrant replacement functor). 
\end{remark}

\subsection{Models for the motivic homotopy theory}\label{motivicspaces2} 

The motivic homotopy category $\HH$ was constructed by Morel and Voevodsky in \cite{MV99}.
Although they worked with an injective local model structure on the category of simplicial sheaves on $\Sm_S$, the motivic homotopy category $\HH$ can be equivalently 
established by performing two left Bousfield localizations on the projective model category $\sPre(\Sm_S)$ of Theorem~\ref{proj-simp} (see also \cite{Blander}).
The first localization of $\sPre(\Sm_S)$ yields the model category $\sPre(\Sm_S)_{\A^1}$ which was already considered in Proposition~\ref{A1project}.
In order to describe the second localization, we recall the definition of a Nisnevich distinguished square.

\begin{definition}
A \emph{Nisnevich distinguished square} is a pullback diagram in $\Sm_S$ 
\begin{equation*}
\alpha = \left(
\vcenter{\xymatrix{
W \ar@{->}[d]\ar@{->}[r]  & Y\ar@{->}[d]^p\\
U\ar@{->}[r]^i & X
}}\right)
\end{equation*}
such that $i$ is an open immersion, $p$ is an \'etale morphism and the induced morphism $p^{-1}((X\setminus i(U))_{\rm red})\xrightarrow{\cong} (X\setminus i(U))_{\rm red}$ 
is an isomorphism.
\end{definition}

For each Nisnevich distinguished square $\alpha$ as above, let $P(\alpha) \to X$ in $\sPre(\Sm_S)$ be the morphism from the pushout $P$ in $\sPre(\Sm_S)$ of the upper part 
$U\leftarrow W\to Y$ of the square to its lower right corner $X$. Here all schemes are identified with the associated representable presheaves.
Consider the set of morphisms 
\begin{equation}\label{thesetNis}
\Nis = \{ (P(\alpha) \to X)\}_{\alpha}
\end{equation}
for each Nisnevich distinguished square $\alpha$ 

Let $\sPre(\Sm_S)_{\A^1,\Nis}$ denote the left Bousfield localization of the model category $\sPre(\Sm_S)_{\A^1}$ at the set $\Nis$.
Following Blander~\cite{Blander}, the model category $\sPre(\Sm_S)_{\A^1,\Nis}$ is Quillen equivalent to the model category of motivic spaces as defined by Morel--Voevodsky in~\cite{MV99}.
We will refer to $\sPre(\Sm_S)_{\A^1,\Nis}$ as the \emph{motivic model category}.
Accordingly, the meaning of  \emph{motivic fibrant} objects, etc., will refer to this particular choice of model category for motivic homotopy theory.
Furthermore, $L_{\mot}$ will denote a fibrant replacement functor for this model structure.

\begin{proposition}\label{motivichashd1}
The motivic model category satisfies the homotopical descent condition \eqref{HD1}.
\end{proposition}
\begin{proof}
This model category is right proper by \cite[Lemma~3.4]{Blander}.
Then the result follows from Theorem~\ref{pdescent}.
\end{proof}

Instead of the two-step left Bousfield localization
\[
\sPre(\Sm_S)\to \sPre(\Sm_S)_{\A^1}\to \sPre(\Sm_S)_{\A^1,\Nis},
\]
we may likewise first localize the objectwise projective model category $\sPre(\Sm_S)$ at the set $\Nis$ from \eqref{thesetNis} to obtain a model category $\sPre(\Sm_S)_{\Nis}$ and afterwards invert the $\A^1$-equivalences.
We will refer to $\sPre(\Sm_S)_{\Nis}$ as the \emph{Nisnevich local model category}.
The functor $L_{\Nis}$ will denote a fibrant replacement functor for this model structure.

We record the following well known theorem whose proof follows from \cite[Thm.~2.2]{Voevodsky2} and \cite[Prop.~2.17]{Voevodsky2} together with \cite[Lemma~4.3]{Blander}.

\begin{theorem} \label{jardinevoevodsky}
The Nisnevich local model category $\sPre(\Sm_S)_{\Nis}$ is the same as the local model structure $\sPre(\C,\tau)$ (see Theorem~\ref{toen-vezzosi}) applied to the Nisnevich site $\Sm_S$.
The left Bousfield localization $\sPre(\Sm_S)\to \sPre(\Sm_S)_{\Nis}$ is homotopy left exact.
\end{theorem}

Yet another model for the motivic homotopy theory was constructed in \cite[Thm.~2.4]{HS}.
This is defined by a model structure on the category $\SPre(\SSm_S)$ which is Quillen equivalent to $\sPre(\Sm_S)_{\A^1, \Nis}$.
More precisely, it is the model structure which is transported from $\sPre(\Sm_S)_{\A^1, \Nis}$ along the adjunction $(\Hs, \Uc)$.
In this model category, which we denote by $\SPre(\SSm_S)_{\Uc \mathrm{mot}}$, a morphism is a \emph{weak equivalence} (respectively, \emph{fibration}) if it is a weak equivalence (respectively, fibration) in $\sPre(\Sm_S)_{\A^1,\Nis}$ after applying the functor $\Uc$.
There is a Quillen equivalence 
\[
\Hs \colon \sPre(\Sm_S)_{\A^1, \Nis} \rightleftarrows \SPre(\SSm_S)_{\Uc \mathrm{mot}} \colon \Uc.
\]
This model category should not be confused with the $\Uc$-local model category $\SPre(\SSm_S)_{\mathrm{\mathcal{U}loc}}$ from Theorem~\ref{local-simp}.

\begin{proposition}\label{prop:equivalenttothemotivic}
The model category $\SPre(\SSm_S)_{\Uc \mathrm{mot}}$ is the same as the left Bousfield localization of the model category $\SPre(\SSm_S)$ from Theorem~\ref{proj-simp} at the set $\Hs(\Nis)$. 
\end{proposition}
\begin{proof}
By \cite[Thm.~3.3.10.1b]{Hirschhorn} and Proposition~\ref{A1project}, there is a Quillen equivalence
\[
\Hs:\sPre(\Sm_S)_{\A^1,\Nis}\leftrightarrows \SPre(\SSm_S)_{\Hs(\Nis)}:\Uc
\]
where the model category on the right-hand side is the left Bousfield localization in question.
As the cofibrations of the model categories $\SPre(\SSm_S)_{\Uc \mathrm{mot}}$ and $\SPre(\SSm_S)_{\Hs(\Nis)}$ are the same, it suffices to show that they have the same fibrant objects.
An object $F\in \SPre(\SSm_S)_{\Uc \mathrm{mot}}$ is fibrant if and only if $\Uc(F)$ is fibrant in $\sPre(\Sm_S)_{\A^1,\Nis}$.
This is the case if and only if $\Uc(F)$ is fibrant in $\sPre(\Sm_S)_{\A^1}$ and
\[
\mathrm{map}(X, \Uc(F)) \to \mathrm{map}(P(\alpha), \Uc(F))
\]
is a weak equivalence for all $P(\alpha) \to X$ in $\Nis$.
As the adjunction~\eqref{adjunction} is a simplicial adjunction, the latter is equivalent to the requirement that the map
\[
\mathrm{map}(\Hs(X), F) \to \mathrm{map}( \Hs(P(\alpha)), F)
\]
is a weak equivalence for all $P(\alpha) \to X$ in $\Nis$. 
But these are exactly the conditions for $F$ to be a fibrant in $\SPre(\SSm_S)_{\Hs(\Nis)}$.
The result follows. 
\end{proof}

\subsection{The motivic homotopy theory is not a model topos}\label{motivicnotatopos}

In this subsection, we provide some details of an argument showing that the motivic homotopy theory is not a model topos.
This was sketched in \cite[Rem.~3.5]{SO12}.

Recall that a simplicial presheaf is called \emph{$\A^1$-local}, if it is $\A^1$-homotopy invariant after a Nisnevich local fibrant replacement (or, in other words, if its fibrant replacement in $\sPre(\Sm_S)_{\Nis}$ is already motivic fibrant).
This property is clearly invariant under Nisnevich local weak equivalences.

\begin{example}
A discrete simplicial presheaf is Nisnevich local fibrant if and only if it is a sheaf.
Hence, a Nisnevich sheaf (considered as a discrete simplicial presheaf) is $\A^1$-local if and only if it is $\A^1$-invariant.
\end{example}

A Nisnevich sheaf of groups $G$ is called \emph{strongly $\A^1$-invariant}, if its classifying space $BG$ is $\A^1$-local (or, in other words, if the Nisnevich cohomology groups $H^0_{\Nis}(-;G)$ and $H^1_{\Nis}(-;G)$ are $\A^1$-invariant).

\begin{example}
Let $S$ be a regular base scheme.
The Nisnevich sheaf of groups $\GG_m$ is clearly $\A^1$-invariant.
It is also strongly $\A^1$-invariant as regular schemes have an $\A^1$-invariant Picard group $\operatorname{Pic}(-)\cong H^1_{\Nis}(-;\GG_m)$.
\end{example}

For a pointed simplicial presheaf $X$ and an integer $n\geq 0$, let $\tilde\pi_n(X)$ be the Nisnevich sheafification of the presheaf $\pi_n(X)$ given by
\[
[S^n\wedge (-)_{+},X]
\]
where the brackets denote hom-sets in the pointed homotopy category of the projective model structure. 

\smallskip

\noindent \textbf{Assumption.} We assume for the rest of the subsection that $S$ is the spectrum of a perfect infinite field.

\begin{theorem}[{Morel \cite{Morelbook}}]\label{theorem:morelstrictly}
Let $X$ be a pointed simplicial presheaf.
Then the sheaf 
\[
\tilde\pi_1(L_{\mot} X)=a_{\Nis}[(-)_+,\Omega L_{\mot} X]
\]
is strongly $\A^1$-invariant.
(Here $a_{\Nis}$ denotes the Nisnevich sheafification functor.)
\end{theorem}
\begin{proof}
See \cite[Thm.~1.9]{Morelbook}.
\end{proof}

The sheaf of groups $\GG_m$ is $\A^1$-invariant.
Hence, so is the free abelian presheaf of groups $\ZZ[\GG_m]$.
Consider the basepoint $1\colon *\to \GG_m$ and the Nisnevich sheaf of abelian groups $\ZZ(\GG_m)=a_{\Nis}(\ZZ[\GG_m]/\ZZ[*])$.

\begin{proposition}[{Choudhury \cite{Choudhury}}] \label{theorem:choudhury}
The Nisnevich sheaf of groups $\ZZ(\GG_m)$ is $\A^1$-invariant 
but not strongly $\A^1$-invariant.
\end{proposition}
\begin{proof}
See \cite[Lemma~4.6]{Choudhury}.
\end{proof}

Combining these results we can now conclude that the motivic homotopy theory cannot be a model topos. 

\begin{proposition}\label{prop:motivicisnotatopos}
The motivic model category is not a model topos. 
\end{proposition}
\begin{proof}
Suppose that the motivic model category $\sPre(\Sm_S)_{\A^1,\Nis}$ is a model topos.
Using \eqref{HD2}, we will show that this implies a weak equivalence $G\simeq \Omega L_{\mot} BG$ for each $\A^1$-invariant Nisnevich sheaf of groups 
$G$ (see Example~\ref{effective groupoids}).
This leads to a contradiction because then we would have isomorphisms of 
sheaves of groups
\[
G\simeq \tilde\pi_0(G)\simeq \tilde \pi_0(\Omega L_{\mot} BG)\simeq \tilde \pi_1(L_{\mot} BG), 
\]
contradicting Theorem~\ref{theorem:morelstrictly} and Proposition~\ref{theorem:choudhury}.\\
Consider the multiplication $m\colon G\times G\to G$ and the simplicial object in $\sPre(\Sm_S)$
\[
\mathcal{B}G_{\bullet}=\left(\xymatrix@C=50pt{\cdots \ar@<-4pt>[r]|-{G\times m_{23}}\ar@<4pt>[r]|-{m_{12}\times G}\ar@<-12pt>[r]|-{pr_{12}}\ar@<12pt>[r]|-{pr_{23}}& G\times G \ar@<-8pt>[r]|-{pr_1}\ar[r]|-m\ar@<8pt>[r]|-{pr_2} & G \ar@<-4pt>[r]\ar@<4pt>[r] & {*}} \right).
\]
This receives a morphism, by projecting away from the first factor in each simplicial degree, from the simplicial object in $\sPre(\Sm_S)$
\[
\mathcal{E}G_{\bullet}=\left(\xymatrix@C=50pt{\cdots \ar@<-4pt>[r]|-{G\times G \times m_{23}}\ar@<4pt>[r]|-{G\times m_{12}\times G}\ar@<-12pt>[r]|-{pr_{012}}\ar@<12pt>[r]|-{m_{01}\times G\times G}& G\times G\times G \ar@<-8pt>[r]|-{pr_{01}}\ar[r]|-{G\times m_{12}} \ar@<8pt>[r]|-{m_{01}\times G} & G\times G \ar@<4pt>[r]|-m\ar@<-4pt>[r]|-{pr_0} & G} \right).
\]
The fiber of the morphism $\mathcal{E}G_{\bullet} \to \mathcal{B}G_{\bullet}$ is the constant simplicial object $G$ in $\sPre(\Sm_S)$.
It is easily verified that for each morphism $[n]\to [m]$ in the simplex category $\Delta$, the diagram
\[
\xymatrix{
\mathcal{E}G_m \ar@{->}[d]\ar@{->}[r] & \mathcal{E}G_n\ar@{->}[d]\\
\mathcal{B}G_m \ar@{->}[r] & \mathcal{B}G_n\\
}
\]
is a pullback in $\sPre(\Sm_S)$.
The corners of this square are motivic fibrant because they are finite products of the $\A^1$-invariant discrete Nisnevich sheaf $G$.
Moreover, the square is a homotopy pullback in the objectwise model category $\sPre(\Sm_S)$ since $\mathcal {E}G_n\to\mathcal{B}G_n$ is a fibration.
We conclude that the square above is also a motivic homotopy pullback for every morphism $[n] \to [m]$.
But then, if property \eqref{HD2} were satisfied, it would follow that the diagram
\begin{equation}\label{eqn:objectwisebutnotmotivichopullback}
\begin{split}
\xymatrix{
G\simeq \mathcal{E}G_0 \ar@{->}[d]\ar@{->}[r] & \hocolim \mathcal{E}G_{\bullet} \ar@{->}[d]\simeq EG\simeq{*}\\
{*}\simeq \mathcal{B}G_0 \ar@{->}[r] & \hocolim\mathcal{B}G_{\bullet}\simeq BG\phantom{\simeq{*}}\\
}
\end{split}
\end{equation}
is a motivic homotopy pullback.
As explained above, this leads to a contradiction. 
\end{proof}

\subsection{Motivic homotopy pullbacks}\label{partialinteractionofsingandnis}

In this subsection, we collect some results on the interaction between Nisnevich fibrant replacement and the $\Sing$-functor, especially in relation with homotopy pullbacks.

Consider the left Bousfield localization $\sPre(\Sm_S)\to\sPre(\Sm_S)_{\Nis}$ from the projective to the Nisnevich local model structure and let $L_{\Nis}$ be a fibrant replacement functor.
Recall that a commutative square $Q$ of simplicial presheaves is a Nisnevich local homotopy pullback if and only if the square $L_{\Nis}(Q)$ is an objectwise homotopy pullback.
As this Bousfield localization is homotopy left exact by Theorem~\ref{jardinevoevodsky}, an objectwise homotopy pullback square $Q$ is also a Nisnevich local homotopy pullback.

Now consider the Bousfield localization $\sPre(\Sm_S) \to  \sPre(\Sm_S)_{\A^1, \Nis}$ to the motivic model structure and let $L_{\mot}$ be a fibrant replacement functor.
Again, a commutative square $Q$ is a motivic homotopy pullback if and only if the square $L_{\mot}(Q)$ is an objectwise homotopy pullback.
However, as this Bousfield localization is not left exact by Proposition~\ref{prop:motivicisnotatopos}, there exists an objectwise homotopy pullback which is not a motivic homotopy 
pullback (see, e.g., Diagram~\eqref{eqn:objectwisebutnotmotivichopullback}).

In this subsection we will identify some objectwise homotopy pullbacks which are also motivic homotopy pullbacks.
We will make use of the notion of an $\A^1$-local simplicial presheaf from the beginning of the previous Subsection~\ref{motivicnotatopos}.

\begin{proposition}\label{onlylowerrightcornera1local}
Let $X\in\sPre(\Sm_S)$ be $\A^1$-local and let
\[
\xymatrix{
Y' \ar@{->}[d]\ar@{->}[r]  & X'\ar@{->}[d]\\
Y\ar@{->}[r] & X
}
\]
be an objectwise homotopy pullback. Then it is also a motivic homotopy pullback. 
\end{proposition}
\begin{proof}

This follows directly from \cite[Lemma~A.3]{JaSymm}.
\end{proof}

\begin{corollary}
Let $X$ be pointed and $\A^1$-local and $Y'\to X'\to X$ an objectwise homotopy fiber sequence. Then it is also a motivic homotopy fiber sequence.
\end{corollary}

In the rest of this subsection, we want to replace the $\A^1$-locality condition in Proposition~\ref{onlylowerrightcornera1local} by a weaker property. We denote by
\[
 i\colon \Sm_S^\aff \hookrightarrow \Sm_S
\]
the full subcategory of affine schemes (in the absolute sense). Precomposition yields the left adjoint $i^*$ of a Quillen adjunction
\begin{equation}\label{adjunction:affines}
 i^*: \sPre(\Sm_S) \rightleftarrows \sPre(\Sm_S^\aff):i_*
\end{equation}
with respect to the projective model structures.
When both sides are Bousfield localized at the Nisnevich local weak equivalences, this adjunction becomes a Quillen equivalence by \cite[Lemma~3.3.2]{AsokHoyoisWendt}.

\begin{definition}\label{pi0a1local}
A simplicial presheaf $X\in\sPre(\Sm_S)$ is called \emph{$\pi_0^\aff$-$\A^1$-local} if its $\pi_0$-presheaf is $\A^1$-invariant on affine smooth schemes after a Nisnevich local fibrant replacement, i.e., if the presheaf $i^*(\pi_0(L_{\Nis}X))$ is $\A^1$-invariant.
\end{definition}

\begin{remark}
The property of $\pi^\aff_0$-$\A^1$-locality is clearly invariant under Nisnevich local weak equivalences.
If a simplicial presheaf is $\A^1$-local, it is also $\pi^\aff_0$-$\A^1$-local.
\end{remark}

\begin{example}[{see the proof of~\cite[Thm.~5.2.3]{AsokHoyoisWendt}}]
Fix an integer $r\geq 1$ and let $\operatorname{Vect}_r\colon \Sm_S^\op\to\sSet$ denote a functorial version of the groupoid $\operatorname{Vect}_r(U)$ of vector bundles on the scheme $U$ of fixed rank $r$.
An objectwise application of the classifying space functor $B$ yields a simplicial presheaf $B\operatorname{Vect}_r\in\sPre(\Sm_S)$ whose $\pi_0$-presheaf assigns to a scheme $U$ the set 
$\pi_0(B\operatorname{Vect}_r)(U)$ of isomorphism classes of vector bundles of rank $r$ over $U$. This presheaf $\pi_0 (B \operatorname{Vect}_r)$ is $\A^1$-invariant on afffine smooth schemes $U$ if $U$ 
satisfies the Bass--Quillen conjecture (e.g., if $S$ is the spectrum of a field). Since $B\operatorname{Vect}_r$ is always Nisnevich local fibrant, it follows, in this case, that it is also $\pi_0^\aff$-$\A^1$-local.
\end{example}

The $\Sing$-functor does not preserve Nisnevich local fibrancy.
In fact, the $\Sing$-functor does not even preserve $\A^1$-locality, even for discrete Nisnevich local fibrant objects, i.e., sheaves of sets \cite[Ex.~3.2.7]{MV99}.
It was an open question (see \cite[Rem.~2.2.9]{AsokMorel} whether $\Sing(-)$ would at least send schemes to $\A^1$-local objects.
This was answered negatively in \cite{Savant}.
However, we have the following partial results in this direction which are instances of the $\pi_*$-Kan condition.

\begin{theorem}[{Asok--Hoyois--Wendt \cite{AsokHoyoisWendt}}]\label{asokhoyoiswendt}
Let $X\in\sPre(\Sm_S)$ be $\pi^\aff_0$-$\A^1$-local.
Then $L_\Nis\Sing(L_{\Nis} X)$ is already motivic fibrant.
\end{theorem}
\begin{proof}
See \cite[Thm.~5.1.3]{AsokHoyoisWendt}.
\end{proof}

There is the following strengthening of \cite[Lemma~4.2.1]{AsokHoyoisWendt}.

\begin{lemma}\label{betterbousfieldfriedlander}
Let
\begin{equation}\label{potentiala1homotopypullbackofbisimplicialsets}
\begin{split}
\xymatrix{
Y' \ar@{->}[d]\ar@{->}[r]  & X'\ar@{->}[d]\\
Y\ar@{->}[r] & X
}
\end{split}
\end{equation}
be a commutative diagram of bisimplicial sets (with indices $p$ and $q$) such that for each $q\geq 0$, the diagram
\begin{equation*}
\begin{split}
\xymatrix{
Y'_q \ar@{->}[d]\ar@{->}[r]  & X'_q\ar@{->}[d]\\
Y_q\ar@{->}[r] & X_q
}
\end{split}
\end{equation*}
is a homotopy pullback of simplicial sets.
If the simplicial set $([q]\mapsto\pi_0(X_q))$ is constant, then the diagonal applied to \eqref{potentiala1homotopypullbackofbisimplicialsets} is a homotopy pullback of simplicial sets.
\end{lemma}
\begin{proof}
This is~\cite[Prop.~5.4]{Rezkpistar}. The statement in op.cit.~uses simplicial spaces: Here $p$ is the `space direction'
and $q$ is the `simplicial direction'. In \cite{Rezkpistar}, a morphism $X'\to X$ of simplicial spaces is called a realization fibration, if the conclusion of the lemma is valid for all commutative diagrams \eqref{potentiala1homotopypullbackofbisimplicialsets} which are homotopy pullbacks in each degree $q$.
\end{proof}

The previous lemma can be used to prove a strengthening of Proposition~\ref{onlylowerrightcornera1local}.
The proof is similar to parts of \cite[Thm.~4.2.3]{AsokHoyoisWendt}.

\begin{theorem}\label{theoremonmotivichomotopypullbacks}
Let $X\in\sPre(\Sm_S)$ be $\pi^\aff_0$-$\A^1$-local and let
\begin{equation}\label{potentiala1homotopypullback}
\begin{split}
\xymatrix{
Y' \ar@{->}[d]\ar@{->}[r]  & X'\ar@{->}[d]\\
Y\ar@{->}[r] & X
}
\end{split}
\end{equation}
be an objectwise homotopy pullback.
Then it is also a motivic homotopy pullback. 
\end{theorem}
\begin{proof}
By homotopy left exactness of the Nisnevich localization functor, we may assume that all objects are Nisnevich local fibrant.
By Theorem~\ref{asokhoyoiswendt} and Proposition~\ref{onlylowerrightcornera1local} it suffices to show that $L_\Nis\Sing(\ref{potentiala1homotopypullback})$ is an objectwise homotopy pullback.  
We know from \cite[Lemma~3.3.2]{AsokHoyoisWendt} that this square $L_\Nis\Sing(\ref{potentiala1homotopypullback})$ is objectwise equivalent to the square $i_* L_{\Nis,\aff} i^* \Sing(\ref{potentiala1homotopypullback})$.
As Quillen right adjoints preserve homotopy pullbacks, it suffices to show that the square $i^* \Sing(\ref{potentiala1homotopypullback})$ is an objectwise homotopy pullback.
In other words, we have to show that $\Sing(\ref{potentiala1homotopypullback})(U)$ is a homotopy pullback square of simplicial sets for every affine scheme $U\in\Sm_S^\aff$.
We fix such a scheme $U$ and consider the diagram
\[
\xymatrix{
Y'_p(\ADelta^q\times U) \ar@{->}[d]\ar@{->}[r]  & X'_p(\ADelta^q\times U)\ar@{->}[d]\\
Y_p(\ADelta^q\times U)\ar@{->}[r] & X_p(\ADelta^q\times U)
}
\]
of bisimplicial presheaves whose diagonal is the square $\Sing(\ref{potentiala1homotopypullback})(U)$ in question.
Now, the simplicial set $[q]\mapsto\pi_0(X(\ADelta^q\times U))$ is constant by assumption.
Hence $\Sing(\ref{potentiala1homotopypullback})(U)$ is an objectwise homotopy pullback by Lemma~\ref{betterbousfieldfriedlander}.
\end{proof}

\begin{corollary}[{{see~\cite[Thm.~2.1.5]{AsokHoyoisWendtII}}}]
Let $X\in\sPre(\Sm_S)$ be pointed and $\pi^\aff_0$-$\A^1$-local and let $Y'\to X'\to X$ be an objectwise homotopy fiber sequence.
Then it is also a motivic homotopy fiber sequence.
\end{corollary}

\begin{corollary}\label{corollaryomegaanda1}
Let $X\in\sPre(\Sm_S)$ be a pointed objectwise fibrant simplicial presheaf which is $\pi^\aff_0$-$\A^1$-local. 
Then
\begin{equation}\label{omegaandla1}
 \Omega L_{\mot}(X) \simeq L_{\mot} \Omega(X).
\end{equation}
\end{corollary}

\begin{remark}
Over the spectrum of a perfect infinite field, Morel showed in \cite[Thm.~6.46]{Morelbook} that for a pointed and stalkwise connected Nisnevich local fibrant $X$, the equivalence \eqref{omegaandla1} holds if and only if the sheaf of groups $\tilde\pi_0(L_{\mot} \Omega X)$ is strongly $\A^1$-in\-variant.
\end{remark}

\section{\texorpdfstring{The $\Uc$-local model category for $\SSm_S$}{The U-local model category for SmS}}\label{sectionfive}

Applying Theorem~\ref{local-simp} to $\SSm_S$ equipped with the Nisnevich topology on $\Sm_S$, denoted $\Nis$, we obtain the $\Uc$-local model category $\SPre(\SSm_S)_{\Uc \Nis}$.
By Theorem~\ref{local-topos}, this model category is a model topos. In this section, we compare this model topos with the model $\SPre(\SSm_S)_{\Uc \mathrm{mot}}$ for the motivic 
homotopy theory from Proposition~\ref{prop:equivalenttothemotivic}. We note that even though we restrict here entirely to the case of the $\Uc$-local model structure associated with $\SSm_S$ with the Nisnevich topology, it may also be interesting to consider the corresponding homotopy theory for other Grothendieck topologies on $\Sm_S$ as well.

\subsection{The model topos $\SPre(\SSm_S)_{\Uc \Nis}$.}
Using Proposition~\ref{tauDelta}, we can identify the Grothendieck topology on $\mathrm{Ho}(\SSm_S)$ that gives rise to the model topos $\SPre(\SSm_S)_{\Uc \Nis}$.
A sieve on $X \in \mathrm{Ho}(\SSm_S)$ 
\[
U \rightarrowtail y_{\mathrm{Ho}(\SSm_S)}(X) 
\]
is a $[\Nis]$-\emph{covering sieve} if
\[
\gamma^*(U) \rightarrowtail \gamma^*(y_{\mathrm{Ho}(\C)}(X))
\]
is an isomorphism after Nisnevich sheafification.
As shown in Lemma~\ref{identification-topologies}, this corresponds to a sieve which is generated by the image of a Nisnevich sieve on $\Sm_S$ under $\gamma \colon \Sm_S \to \mathrm{Ho}(\SSm_S)$.
Let $[\Nis]$ denote the collection of $[\Nis]$-covering sieves in $\mathrm{Ho}(\SSm_S)$.
The following proposition is a special case of Proposition~\ref{tauDelta}.

\begin{proposition}
The collection of sieves $[\Nis]$ defines a Grothendieck topology on $\mathrm{Ho}(\SSm_S)$. 
The model topos $\SPre(\SSm_S, [\Nis])$ (see Theorem~\ref{toen-vezzosi}) is the same as $\SPre(\SSm_S)_{\Uc \Nis}$. 
\end{proposition}

\begin{remark}(Naturality revisited)
We observed in Remark~\ref{remark:naturality} that for two small simplicial categories $\C$ and $\C'$ whose underlying categories $\C_0$ and $\C_0'$ are equipped with Grothendieck 
topologies $\tau$ and $\tau'$ and a simplicial functor $F\colon \C\to\C'$ restricting to a morphism of sites, the adjunction $(F_!,F^*)$ is not necessarily a Quillen adjunction for 
the $\Uc$-local model structures.
However, this is true in the following special case:
Let $f\colon R\to S$ be a morphism of noetherian schemes of finite Krull dimension.
The pullback functor $F\colon \SSm_S\to \SSm_R$, $U\mapsto U\times_S R$ is a simplicial functor inducing a Quillen adjunction
\[
f^{\Delta,*}\colon\sPre^\Delta(\SSm_S)\leftrightarrows \sPre^\Delta(\SSm_R)\colon f_*^\Delta 
\]
between the projective model categories (it is common to write $(f^*,f_*)$ for $(F_!,F^*)$).
The right adjoint is given by $f_*^\Delta (G)=G(F(-))=G(-\times_S R)$ and the left adjoint is determined via enriched left Kan extension by $\mathrm{map}_S(-,U)\mapsto \mathrm{map}_R(-,U\times_S R)$.

The pair $(f^{\Delta,*}, f_*^\Delta)$ is a Quillen adjunction for the $\Uc$-local model structures if $f^{\Delta,*}$ preserves $\Uc$-local weak equivalences. 
This holds if the diagram
\[
\xymatrix{
\sPre^\Delta(\SSm_S) \ar@{->}[d]_\Uc\ar@{->}[r]^-{f^{\Delta,*}}& \sPre^\Delta(\SSm_R)\ar@{->}[d]^\Uc\\
\sPre(\Sm_S)\ar@{->}[r]^-{f^{*}}& \sPre(\Sm_R)
}
\]
commutes (for which it suffices to check only on representables $\mathrm{map}_S(-,U)$) and $f^*$ preserves Nisnevich local weak equivalences. In the case of a \emph{smooth} morphism $f$, this diagram commutes since we have an isomorphism 
\[
f^*(\mathrm{map}_S(-,U))\cong \mathrm{map}_R(-,U\times_S R)
\]
and $f^*$ is both a left and a right Quillen functor. 
\end{remark} 

The model topos $\SPre(\SSm_S)_{\Uc \Nis}$ can also be modelled in terms of non-enriched simplicial presheaves. This can be done by trasporting the 
$\Uc$-local model structure of $\SPre(\SSm_S)_{\Uc \Nis}$ to $\sPre(\Sm_S)_{\A^1}$ along the Quillen equivalence of Proposition~\ref{A1project}.
We call a morphism $\eta\colon F \to G$ in $\sPre(\Sm_S)$ a \emph{Sing-Nisnevich local weak equivalence} if the induced morphism $\Sing(F) \to \Sing(G)$ is a Nisnevich local weak equivalence. 

\begin{theorem} \label{theorem:singlocalmodel}
There is a left proper simplicial combinatorial model structure 
\[
\sPre(\Sm_S)_{\Sing\mathrm{\text{-}}\Nis} 
\]
on the category $\sPre(\Sm_S)$ where the cofibrations are the projective cofibrations and the weak equivalences are the Sing-Nisnevich local weak equivalences.
This is a left Bousfield localization of the model category $\sPre(\Sm_S)_{\A^1}$ and a model topos. The adjunction 
\[
\Hs: \sPre(\Sm_S)_{\Sing\mathrm{\text{-}}\Nis} \rightleftarrows \SPre(\SSm_S)_{\Uc \Nis}: \Uc 
\]
is a Quillen equivalence.
\end{theorem}
\begin{proof}
This is induced from the Quillen equivalence of Proposition~\ref{A1project} after localizing the right-hand side at the $\Uc$-local weak equivalences and the left-hand side at the inverse image of this class under the (derived) functor $\Hs$, that is, 
the $\Sing$-Nisnevich local weak equivalences. Note that this class contains the weak equivalences of $\sPre(\Sm_S)_{\A^1}$. The class of $\Uc$-local weak equivalences is accessible and accessibly embedded because it is the class of weak equivalences of a combinatorial model category \cite[Cor.~A.2.6.8]{htt}, \cite{RR1}.
Therefore its inverse image under the accessible functor $\Hs$ is also accessible and accessibly embedded.
This shows the existence of the left Bousfield localization $\sPre(\Sm_S)_{\Sing\mathrm{\text{-}}\Nis}$, using \cite[Prop.~A.2.6.10]{htt}. The Quillen equivalence is an immediate consequence of Proposition \ref{A1project}. 
\end{proof} 

\begin{remark}
A fibrant replacement functor for the model topos $\sPre(\Sm_S)_{\Sing\mathrm{\text{-}}\Nis}$ of the previous Theorem~\ref{theorem:singlocalmodel} is \emph{not} given by the functor $L_{\Nis}\Sing(-)$ as the latter is not $\A^1$-invariant in general (see, e.g., \cite[Ex.~3.2.7]{MV99}).
\end{remark}

\subsection{Comparison with $\SPre(\SSm_S)_{\Uc \mathrm{mot}}$} \label{comparison-with}
First, we observe that there is a left Bousfield localization
\[
\id\colon \SPre(\SSm_S)_{\Uc \Nis} \to \SPre(\SSm_S)_{\Uc \mathrm{mot}}.
\]
This Quillen adjunction however is not a Quillen equivalence since the left-hand side is a model topos (see Theorem~\ref{local-topos}) while the right-hand side is not (see Proposition~\ref{prop:motivicisnotatopos}).
Similarly, the comparison between these two homotopy theories, represented by $\SPre(\SSm_S)_{\Uc \Nis}$ and $\SPre(\SSm_S)_{\Uc \mathrm{mot}}$ respectively, can also be studied on the 
`non-enriched side' using the left Bousfield localization 
$$\mathrm{id} \colon \sPre(\Sm_S)_{\Sing\mathrm{\text{-}}\Nis} \to \sPre(\Sm_S)_{\A^1, \Nis}.$$
Note that neither of these two left Quillen functors is homotopy left exact since the motivic homotopy theory is not a model topos.

\begin{example}
We give an example of a motivic weak equivalence which is not a $\Sing$-Nisnevich local weak equivalence. Consider the Nisnevich sheaf of groups $G=\ZZ(\GG_m)$ from Proposition~\ref{theorem:choudhury} and the motivic weak equivalence
\[
   f\colon L_{\Sing\mathrm{\text{-}}\Nis} B G \to  L_\mot  L_{\Sing\mathrm{\text{-}}\Nis} BG                                                                                                                                                                                                       
\]
where $L_{\Sing\mathrm{\text{-}}\Nis}$ is a fibrant replacement functor for the model topos $\sPre(\Sm_S)_{\Sing\mathrm{\text{-}}\Nis}$ of Theorem~\ref{theorem:singlocalmodel}.
Consider the 
canonical commutative triangle
\[
\xymatrix@C=2ex{
& G \ar@{->}[dl]\ar@{->}[dr]&\\
\Omega L_{\Sing\mathrm{\text{-}}\Nis} B G \ar@{->}[rr]^-{\Omega(f)}&& \Omega L_\mot L_{\Sing\mathrm{\text{-}}\Nis} BG.
}
\]
The left diagonal morphism is a $\Sing$-Nisnevich local weak equivalence since the model category $\sPre(\Sm_S)_{\Sing\mathrm{\text{-}}\Nis}$ is a model topos.
Hence it is also a motivic weak equivalence.
We observed in the proof of Proposition~\ref{prop:motivicisnotatopos} that the right diagonal morphism is not a motivic weak equivalence.
Therefore, also $\Omega(f)$ cannot be a motivic equivalence.
This implies that $f$ cannot be a $\Sing$-Nisnevich local weak equivalence.
\end{example}

The comparison between $\SPre(\Sm_S)_{\Sing\mathrm{\text{-}}\Nis}$ and the motivic homotopy theory is essentially about the question of how much of Nisnevich descent is encoded 
in the $\Uc$-local model structure. We discuss the comparison between the $\sPre(\Sm_S)_{\Nis}$ and the model category $\sPre(\Sm_S)_{\Sing\mathrm{\text{-}}\Nis}$ and then 
identify the descent condition in question based on the results of Section \ref{local-model-structures}.

\begin{proposition}\label{singdoesnotpreserve}
Let $S$ be a regular scheme. The functor $\Sing\colon \sPre(\Sm_S) \to \sPre(\Sm_S)$ does not preserve Nisnevich local weak equivalences.
\end{proposition}
\begin{proof}
Since for a scheme $\operatorname{Spec}(A)\in \Sm_S$ the units of the ring $A$ are the same as the units of the ring $A[T_0,\ldots,T_n]$, we have an isomorphism
\[
\hom_{\sPre}(U,\A^1\setminus\{0\})\cong   \hom_{\sPre}(U\times \ADelta^n,\A^1\setminus\{0\})
\]
for every $n\geq 0$.
Therefore $\A^1\setminus\{0\}\cong \Sing(\A^1\setminus\{0\})$ and likewise $\A^1\setminus\{1\}\cong \Sing(\A^1\setminus\{1\})$ and $\A^1\setminus\{0,1\}\cong\Sing(\A^1\setminus\{0,1\})$.
Consider the Zariski distinguished square
\[
\xymatrix{
\A^1\setminus\{0,1\} \ar@{->}[d]_-g\ar@{->}[r]^-f  & \A^1\setminus\{0\}\ar@{->}[d]\\
\A^1\setminus\{1\}\ar@{->}[r] & \A^1
}
\]
and let $P$ be the pushout of $f$ and $g$ in $\sPre(\Sm_S)$. The induced morphism $P\to \A^1$ is a Nisnevich local weak equivalence.
The $\Sing$-functor preserves all limits and colimits, therefore $\Sing(P)$ is the pushout of $\Sing(f)$ and $\Sing(g)$.
Since $\Sing(f)$ is a monomorphism, $\Sing(P)$ is also the homotopy pushout in $\sPre(\Sm_S)_{\Nis}$ and therefore $P\to\Sing(P)$ is a Nisnevich local weak equivalence. 

Suppose that the $\Sing$-functor preserves all Nisnevich local weak equivalences between cofibrant objects. 
Then $\Sing(P)\to\Sing(\A^1)$ is a Nisnevich local weak equivalence and hence $\A^1\to\Sing(\A^1)$ is a Nisnevich local weak equivalence.
This is a contradiction since $\Sing(\A^1)$ is objectwise contractible by \cite[Cor.~1.6]{HS} and therefore also contractible in the Nisnevich local model structure.
But this is not the case for $\A^1$, which is the contradiction.
Therefore the $\Sing$-functor does not preserve Nisnevich local weak equivalences.
\end{proof}

\begin{corollary}\label{notanadjunction}
The adjunction $\Hs: \sPre(\Sm_S)_{\Nis} \rightleftarrows \SPre(\SSm_S)_{\mathcal{U}\Nis}: \mathcal{U}$ is not a Quillen adjunction.
In particular, $\Hs$ does not send Nisnevich squares to homotopy pushouts in general.
\end{corollary}
\begin{proof}
The functor $\Sing\simeq \mathcal{U}\Hs$ does not preserve Nisnevich local weak equivalences between cofibrant objects by the proof of Proposition~\ref{singdoesnotpreserve}.
\end{proof}

\begin{remark}
An alternative proof of Corollary~\ref{notanadjunction} is given as follows.
Let $F\in \SPre(\SSm_S)_{\mathcal{U}\Nis}$ be a fibrant simplicial presheaf.
If the functor $\mathcal{U}$ to the Nisnevich local model category were a right Quillen functor, $\Uc (F)$ would be Nisnevich local fibrant.
This implies that $\Uc (F)$ is motivic fibrant since it is also $\A^1$-invariant.
However, $\SPre(\SSm_S)_{\mathcal{U}\Nis}$ is not a model for the motivic homotopy theory (see, e.g., Proposition~\ref{prop:motivicisnotatopos}).
\end{remark}

Another way of comparing $\SPre(\SSm_S)_{\Uc \Nis}$ with the motivic homotopy theory is obtained from the functor $\Uc$ regarded as a \emph{left} Quillen functor 
(see Remark \ref{U-left-Quillen}):
$$\Uc \colon \SPre(\SSm_S)_{\Uc \Nis} \rightleftarrows \sPre^{\mathrm{inj}}(\Sm_S)_{\Nis} \colon \mathcal{G}.$$
Here $\sPre^{\mathrm{inj}}(\Sm_S)_{\Nis}$ denotes the injective local model structure where the cofibrations are the monomorphisms (see \cite{Ja}) and 
$\mathcal{G}$ denotes the right adjoint. More expicitly, given $F \in \sPre^{\mathrm{inj}}(\Sm_s)_{\Nis}$ and $U \in \SSm_S$, the right Kan extension 
$\mathcal{G}$ is defined as an end by the formula
\[
\mathcal{G}(F)(U) = \mathrm{map}_{\sPre(\Sm_S)}(\Uc \big(\mathrm{map}(U, -)\big), F(-)).
\]
Composing this with the Bousfield localization at the class of $\A^1$-equivalences, we obtain a Quillen adjunction 
$$\Uc \colon \SPre(\SSm_S)_{\Uc \Nis} \rightleftarrows \sPre^{\mathrm{inj}}(\Sm_S)_{\A^1, \Nis} \colon \mathcal{G}.$$
As a consequence, we have the following way of constructing $\Uc$-local fibrant objects (cf. Remark \ref{cover-reflecting}).

\begin{proposition}
Let $F \in \sPre^{\mathrm{inj}}(\Sm_S)_{\Nis}$ be a fibrant object. Then $\mathcal{G}(F)$ is fibrant in $\SPre(\SSm_S)_{\Uc \Nis}$.
\end{proposition}

The comparison between $\SPre(\SSm_S)_{\Uc \mathrm{mot}}$ and $\SPre(\SSm_S)_{\Uc \Nis}$ can be specified further by identifying an explicit set of morphisms which defines this Bousfield localization.
To describe this, it will be convenient to pass to the associated presentable $\infty$-categories and use the $\infty$-categorical notion of a covering sieve as considered by Lurie \cite{htt}. 

Let $\SSm^{\infty}_S$ denote the $\infty$-category associated with the simplicial category $\SSm_S$.
Explicitly, this is given by applying the coherent nerve functor to a fibrant replacement of $\SSm_S$.
Then, the $\infty$-category of presheaves $\mathcal{P}(\SSm^{\infty}_S)$ is equivalent to the presentable $\infty$-category associated with $\SPre(\SSm_S)$ \cite[Prop.~4.2.4.4]{htt}. 
Let $\Hs^{\infty}(\Nis)$ denote the set of morphisms in $\mathcal{P}(\SSm^{\infty}_S)$ that corresponds to $\Hs(\Nis)$.
This is defined by morphisms of presheaves as follows
\begin{equation} \label{infty-Nis}
y_{\SSm^{\infty}_S}(U) \cup_{y_{\SSm^{\infty}_S}(W)} y_{\SSm^{\infty}_S}(Y) \longrightarrow y_{\SSm^{\infty}_S}(X)
\end{equation}
for every Nisnevich distinguished square
\begin{equation*}
\alpha = \left(
\vcenter{
\xymatrix{
W \ar@{->}[d]\ar@{->}[r]  & Y\ar@{->}[d]^p\\
U\ar@{->}[r]^i & X
}}
\right)
\end{equation*}
where $y = y_{\SSm_S^{\infty}}$ denotes the Yoneda embedding.

Following Proposition~\ref{prop:equivalenttothemotivic}, the localization of $\mathcal{P}(\SSm^{\infty}_S)$ at $\Hs^{\infty}(\Nis)$ is equivalent to the presentable $\infty$-category, denoted $\mathcal{P}(\SSm^{\infty}_S)_{\mathrm{mot}}$, associated with the motivic model category $\SPre(\SSm_S)_{\Uc \mathrm{mot}}$.
We may factorize the morphism in \eqref{infty-Nis} into an effective epimorphism followed by a monomorphism (see~\cite[6.2.3]{htt}): 
\[
y(U) \cup_{y(W)} y(Y) \stackrel{Q_{\alpha}}{\twoheadrightarrow} [y(U) \cup_{y(W)} y(Y)] \stackrel{J_{\alpha}}{\rightarrowtail} y(X). 
\]
The collection of monomorphisms $J_{\alpha}$ in $\mathcal{P}(\SSm_S^{\infty})$ that arises this way, for \emph{every} Nisnevich covering sieve $\alpha$, can be identified with the Grothendieck topology on $\SSm^{\infty}_S$ associated with the Grothendieck topology $[\Nis]$ on $\mathrm{Ho}(\SSm^{\infty}_S)$ 
(see \cite[Rem.~6.2.2.3]{htt}).
Indeed, if $\alpha$ is generated by a collection of maps $\{U_{\alpha, i} \to X\}$, then the monomorphism $J_{\alpha}$ corresponds to the ($\infty$-)sieve on $X \in \SSm^{\infty}_S$ that is generated by the same maps (see Remark \ref{difference-of-sheaves}).
Let $[\Nis_{\infty}]$ denote the collection of monomorphisms $J_{\alpha}$ that are obtained this way.
Every $[\Nis_{\infty}]$-local equivalence in $\mathcal{P}(\SSm^{\infty}_S)$ is also an equivalence in the (hypercomplete) $\infty$-topos $\mathcal{P}(\SSm^{\infty}_S)_{[\Nis_{\infty}]}$ 
that is associated with the model topos $\SPre(\SSm_S, [\Nis])$.
As a consequence, the motivic $\infty$-category 
\[
\mathcal{P}(\SSm^{\infty}_S)_{\mathrm{mot}} \simeq \mathcal{P}(\SSm^{\infty}_S)[\Hs^{\infty}(\Nis)^{-1}]
\]
is the localization of the $\infty$-topos $\mathcal{P}(\SSm^{\infty}_S)_{[\Nis_{\infty}]}$ at the set of morphisms 
$$\{Q_{\alpha} \colon \alpha \ \text{Nisnevich covering sieve} \}.$$

\subsection{Summary} \label{summary}
We summarize the connections between the different model categories and Quillen adjunctions in the following diagram. 
{\scriptsize
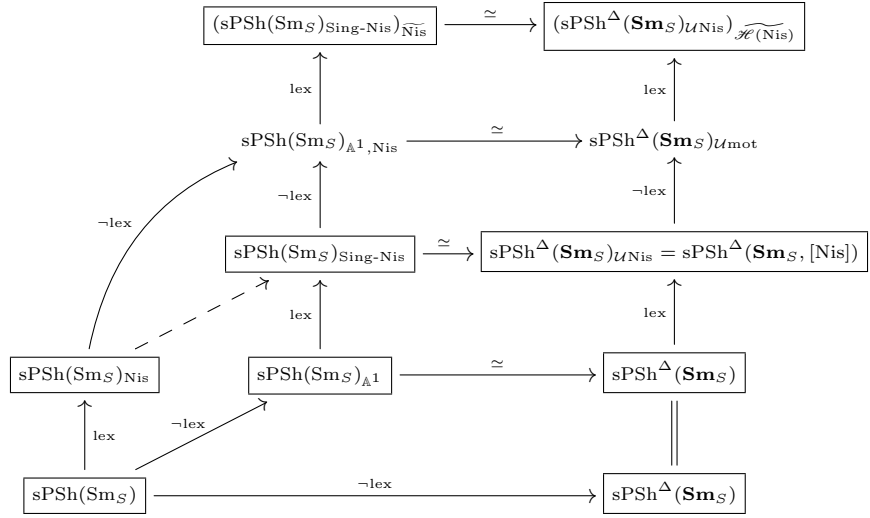
\[
\xymatrix@C=3.5ex{
&\boxed{(\sPre(\Sm_S)_{\Sing\mathrm{\text{-}}\Nis}){}_{\widetilde{\Nis}}}\ar@{->}[r]^-{\simeq} & \boxed{(\SPre(\SSm_S)_{\Uc\Nis}){}_{\widetilde{\Hs(\Nis)}}}\\
&\sPre(\Sm_S)_{\A^1, \Nis}\ar@{->}[r]^\simeq \ar@{->}[u]^-{\text{lex}}& \SPre(\SSm_S)_{\Uc \mathrm{mot}}\ar@{->}[u]^-{\text{lex}}\\
&\boxed{\sPre(\Sm_S)_{\Sing\mathrm{\text{-}}\Nis}}\ar@{->}[r]^(.35)\simeq \ar@{->}[u]^-{\text{$\neg$lex}} & \boxed{\SPre(\SSm_S)_{\Uc\Nis} = \SPre(\SSm_S, [\Nis])}\ar@{->}[u]^-{\text{$\neg$lex}}\\
\boxed{\sPre(\Sm_S)_{\Nis}}\ar@/^2.0pc/@{->}[uur]^-{\text{$\neg$lex}}\ar@{-->}[ur] & \boxed{\sPre(\Sm_S)_{\A^1}}\ar@{->}[u]^-{\text{lex}}\ar@{->}[r]^\simeq & \boxed{\SPre(\SSm_S)}\ar@{->}[u]^-{\text{lex}}\\
\boxed{\sPre(\Sm_S)}\ar@{->}[rr]^-{\text{$\neg$lex}}\ar@{->}[ur]^-{\text{$\neg$lex}}\ar@{->}[u]_-{\text{lex}} && \boxed{\SPre(\SSm_S)} \ar@{=}[u]
}
\]}
The boxes indicate that the corresponding model categories are model topoi. The label `lex' (respectively,~`$\neg$lex') means that the left Quillen functor is homotopy left exact (respectively,~`not homotopy left exact').
The second row consists of models for the motivic homotopy theory. The top row is obtained by applying Theorem~\ref{thmconstructingmodeltopoi} to $\sPre(\Sm_S)_{\Sing\mathrm{\text{-}}\Nis}$ and $\SPre(\SSm_S)_{\Uc\Nis}$ and the respective classes of motivic weak equivalences.
The dotted arrow is not a Quillen adjoint by Corollary~\ref{notanadjunction}.

We remark that for purely formal reasons every functor which is (homotopically) representable 
in the motivic homotopy theory $\SPre(\SSm_S)_{\Uc \mathrm{mot}}$, it is also representable in the $\Uc$-local homotopy theory $\SPre(\SSm_S)_{\Uc\Nis}$. In addition, if it descends to the homotopy theory $\big(\SPre(\SSm_S)_{\Uc \Nis}\big)_{\widetilde{\Hs(\Nis)}}$, then it will again be representable there.

\end{document}